\documentclass{article} 
\usepackage{iclr2027_conference,times}
\iclrfinalcopy

\usepackage{natbib}
\usepackage{booktabs}

\usepackage{amsmath,amsthm,amssymb}
\usepackage{bm}
\usepackage{algorithm}
\usepackage{algpseudocode}
\usepackage{dsfont}

\theoremstyle{plain}
\newtheorem{theorem}{Theorem}[section]
\newtheorem{lemma}[theorem]{Lemma}
\newtheorem{corollary}[theorem]{Corollary} 
\newtheorem{proposition}[theorem]{Proposition}

\theoremstyle{definition}
\newtheorem{definition}[theorem]{Definition}

\newtheorem{problem}[theorem]{Problem}

\theoremstyle{remark}
\newtheorem{remark}[theorem]{Remark}

\usepackage{graphicx}
\usepackage{subcaption}
\usepackage{xcolor,hyperref}
\definecolor{darkblue}{rgb}{0.0,0.0,0.6}
\hypersetup{colorlinks,breaklinks,linkcolor=darkblue,urlcolor=darkblue,anchorcolor=darkblue,citecolor=darkblue}

\title{Cost-Sensitive Online Window Size Selection for Portfolio Management}

\author{Yi-Chen Liu\thanks{
        Interdisciplinary Program of Management and Technology, National Tsing Hua University
    } \; and 
    Chung-Han Hsieh\thanks{ 
        Department of Quantitative Finance, National Tsing Hua University, Taiwan. (\href{mailto:ch.hsieh@mx.nthu.edu.tw}{ch.hsieh@mx.nthu.edu.tw})
       }
    }
       
\date{}

\providecommand{\keywords}[1]
{
  \small	
  \textbf{\textit{Keywords---}} #1
}

\begin{document}

\maketitle

\begin{abstract}
    This paper investigates cost-sensitive online window size selection for portfolio management under changing market conditions. 
    Specifically, we propose a two-level framework that constructs portfolios using candidate window sizes and dynamically aggregates them through online learning. 
    By treating candidate window sizes as ``experts,'' we dynamically update their aggregation weights using turnover-inclusive losses. 
    Moreover, we derive finite-horizon cost-sensitive tracking-regret bounds that account for turnover of the aggregated portfolio, with static regret as a special case. 
    Under bounded losses and cost rates, suitably tuned Fixed Share achieves asymptotically no tracking regret for sublinear switching budgets, with Hedge covering the static case.
\end{abstract}


\section{Introduction}
A fundamental issue in financial forecasting is how much historical information to retain. 
Long windows can smooth short-term noise and capture persistent patterns, whereas short windows can respond more rapidly to recent information.
This trade-off is particularly important in rolling portfolio management, where the choice of estimation window directly affects portfolio decisions.
When portfolios constructed from different estimation windows are combined online, changing their aggregation weights can induce turnover even if the individual portfolios remain unchanged.
Hence, each window expert's transaction costs do not, by themselves, account for the cost of the aggregate portfolio.
This raises the question of how to track changing expert performance while controlling regret that includes the turnover of the aggregate portfolio.

In particular, financial time-series forecasting often involves complex and time-varying dynamics, making the choice of estimation window especially consequential; see, e.g., \cite{peters1994fractal}
Historically, window sizes in financial applications were often chosen empirically or based on forecasting performance. 
For example, \cite{molodtsova2009out} used a fixed 10-year rolling window for monthly exchange-rate forecasting, and \cite{stock2007has} used quarterly units to predict inflation. 
\cite{pesaran2007selection} proposed a forecaster that weighted averages with different window sizes to mitigate model uncertainty, a concept detailed in \cite{Pesaran2011forecast} and finalized with optimal weighting in \cite{pesaran2013optimal}. 
\cite{rossi2012out} generally discussed the topic and provided a statistically robust window selection method. 
\cite{inoue2017rolling} identified the solution for window size that is asymptotically equivalent to minimizing Mean Squared Forecast Error (MSFE), which is theoretically unsolvable. 
However, these approaches do not formulate window-size selection as a sequential learning problem that adapts to evolving markets with performance guarantees.

Dynamic window selection has also been considered in financial prediction and portfolio problems.
A common approach is to evaluate several candidate windows and select the best-performing one. 
For example, \cite{jeon2017time} dynamically weighted different sub-window sizes via probability distributions to predict correlations between financial instruments, although the resulting computational burden limits frequent updating. 
\cite{rajabi2022mlp} used multi-layer perceptrons to determine dynamic windows for Bitcoin prediction, but focused on very short time horizons. 
Thus, a general framework for online window-size adaptation with rigorous performance guarantees remains lacking.

Parallel to this literature, online portfolio selection has developed extensively since the Universal Portfolio of \cite{cover1991universal}, itself motivated by Kelly's growth-optimal betting framework \cite{kelly1956new}. 
Follow-the-Winner methods favor assets or portfolios that have performed well historically. 
Examples include Exponential Gradient (EG) \cite{helmbold1998line}, the Aggregating Algorithm \cite{vovk1998universal}, and Variable Rebalanced Portfolios (VRP) \cite{gaivoronski2000stochastic}. 
By contrast, Follow-the-Loser methods exploit mean reversion, including Passive Aggressive Mean Reversion (PAMR) \cite{li2012pamr}, Confidence Weighted Mean Reversion (CWMR) \cite{li2013confidence}, and OLMAR \cite{li2014online}.

However, how to choose the window size, an important parameter for time series prediction, is less emphasized in online portfolio selection. \cite{gaivoronski2000stochastic} points out the need to consider sliding window size.
\cite{gaivoronski2003line} studies adaptive portfolio selection with transaction costs.
\cite{li2012olmar} propose a window-size-sensitive model OLMAR for trading the portfolio. 
However, that work does not establish a tracking-regret bound that accounts for turnover of the aggregated portfolio.

We formulate window-size adaptation as the online aggregation of sliding-window portfolio experts within the prediction-with-expert-advice framework~\cite{cesa2006prediction}.
Each candidate window generates a portfolio, and the algorithm combines these portfolios using weights updated from the experts' turnover-inclusive losses.
We employ Fixed Share~\cite{herbster1998tracking} to track changing expert performance, with Hedge~\cite{littlestone1994information}
as the static-regret baseline.

Our main contribution is a finite-horizon tracking-regret bound that accounts for turnover of the deployed aggregate portfolio.
The comparator is the best expert sequence under a prescribed switching budget, with each selected expert evaluated by its own turnover-inclusive loss.
The resulting bound makes the dependence on the transaction-cost rate and switching budget explicit, yields a cost-dependent learning-rate choice, and recovers the static-regret bound for Hedge as a special case.

\section{Problem Formulation} 
    Consider a portfolio $\mathcal{M}:=\{1,\dots, m\}$ with $m\geq2$ assets, including a risk-free asset. 
    Let $S_i(t)>0$ denote the adjusted closing price of asset $i\in\mathcal M$ at time $t\ge1$. 
    For asset $i\in\mathcal{M}$, its return at time $t$ is denoted by 
    \[
    y_{i}(t):=\frac{S_{i}(t+1)-S_{i}(t)}{S_{i}(t)}.
    \]
    Let $\mathbf{y}(t):=[y_{1}(t),\dots,y_{m}(t)]^\top\in\mathbb{R}^m$ be the return vector for $m$ assets at time $t\geq 1$.

\subsection{Sliding-Window Expert and Cost-Sensitive Mean-Variance Model}
We extend online window-size selection within the expert-advice framework \cite{cesa2006prediction, orabona2026modern} to account for turnover induced by both expert portfolio updates and changes in aggregation weights.
Let~$\mathcal{N}$ be a finite set of $n\geq 2$ candidate window sizes, where each $j \in \mathcal{N}$ satisfies $j\geq 2$ and is treated as an \emph{expert}. 
For expert $j\in\mathcal{N}$, consider Markowitz's Mean-Variance (MV) optimization problem with turnover costs to obtain the advice $\mathbf{w}_{j}(t) \in \mathcal{W}$; the weight vector optimizes the last $j$-days of data prior to time $t$; see \cite{markowitz1952portfolio,luenberger2013investment}. 
Given an initial history of $\max\mathcal N$ returns, define the rolling estimators for expert $j\in\mathcal N$~by
\begin{align*}
    \widehat{\bm{\mu}}_j(t)
    &:=
    \frac{1}{j}\sum_{s=t-j}^{t-1}\mathbf{y}(s)
    \in \mathbb{R}^m,\\
    \widehat{\Sigma}_j(t)
    &:=
    \frac{1}{j-1}
    \sum_{s=t-j}^{t-1}
    \bigl(\mathbf{y}(s)-\widehat{\bm{\mu}}_j(t)\bigr)
    \bigl(\mathbf{y}(s)-\widehat{\bm{\mu}}_j(t)\bigr)^\top
    \in\mathbb{S}_+^m.
\end{align*}
Here $\mathbb S_+^m$ denotes the set of real symmetric positive semidefinite $m\times m$ matrices.

\begin{problem}[Cost-Sensitive Mean-Variance Optimization]
\label{problem: cost-sensitive mv problem}
    Let $\mathbf{w}_j(t-1) \in \mathcal{W}$ denote the portfolio weights held by expert $j$ immediately before rebalancing at time $t$, with $\mathbf{w}_j(0):=\mathbf{w}_0$ for a prescribed initial portfolio $\mathbf{w}_0 \in \mathcal{W}$, and let $c(t) \geq 0$ be the proportional transaction cost rate. 
The \emph{cost-sensitive mean-variance} model for expert $j$ can be written as 
    \begin{align} \label{problem: cost-sensitive MV}
        \max_{\mathbf{w} \in \mathcal{W}} \;\widehat{\bm{\mu}}_j(t)^\top\mathbf{w}
        -
        \frac{1}{2} \mathbf{w}^\top\widehat{\Sigma}_j(t)\mathbf{w}
        - 
        c(t) \|\mathbf{w} - \mathbf{w}_j(t-1)\|_1, \qquad j \in \mathcal{N}
    \end{align}
where
    $
    \mathcal{W} 
    := \left\{
        \mathbf{w} \in \mathbb{R}_+^m : \mathbf{w}^\top \mathbf{1} = 1 \right\}.
    $
Here $\mathbf1$ denotes the all-ones vector of the appropriate dimension.
\end{problem}    

Let $\mathbf{w}_{j}(t)$, the optimal advice solving the problem for $j$-day data at time $t$, be the advice for expert $j$. 
Note that Problem~\ref{problem: cost-sensitive mv problem}  is a concave program, which can be solved efficiently by a standard solver such as CVXPY; see \cite{diamond2016cvxpy}.

After obtaining the expert advice, we assign an \emph{aggregation weight vector} in the simplex:
$$
    \mathbf{q}(t) \in \Delta_n 
    := 
    \left\{ 
        \mathbf{q}\in\mathbb{R}_+^n:
        \sum_{j\in\mathcal{N}}q_j=1
    \right\}
    \subseteq \mathbb{R}^n
    $$
and form the \emph{aggregateportfolio}
\[
\widehat{\mathbf{w}}(t)
:=
\sum_{j\in\mathcal{N}}q_j(t)\mathbf{w}_j(t).
\]
Thus, the method dynamically aggregates the window experts rather than necessarily selecting a single window. 
Here $\mathbf w_j(t)$ and $\widehat{\mathbf w}(t)$ are portfolio weight vectors over assets, whereas $\mathbf q(t)$ is the aggregation weight vector over window experts.

\subsection{Tracking Regret}
\label{subsec: regret analysis}
Consider a sequential decision-making framework over a discrete time horizon $t=1,2, \dots$. 
Let
    $
    \ell:\mathcal W\times\mathcal Y\to\mathbb R
    $
be a loss function, assumed to be convex in its first argument. 
At each time $t$, the algorithm first selects a decision  
    $
    \widehat{\mathbf{w}}(t)  \in \mathcal{W}
    $. 
After the decision is made, the market outcome $\mathbf{y}(t) \in \mathcal{Y} \subseteq \mathbb{R}^m$ is revealed, and the algorithm incurs the prediction loss
    $
    \ell(\widehat{\mathbf{w}}(t), \mathbf{y}(t)).
    $
    
Recall that each $j \in \mathcal{N}$ indexes the expert associated with a $j$-day sliding window; at each time $t$, this expert provides advice $\mathbf{w}_{j}(t) \in \mathcal{W}$.
For a finite horizon $T \geq 1$, define the set of all expert sequences over $T$ periods~as
    \[ 
    \mathcal{N}^T := \underbrace{\mathcal N\times\cdots\times\mathcal N}_{T\ \mathrm{times}}. 
    \] 
Write
    $ 
    \mathbf{j}:=(j_1,\dots,j_T)\in\mathcal N^T 
    $
for an arbitrary comparator sequence of experts, where $j_t$ is the expert selected by the comparator at time $t$.
For an expert sequence $\mathbf{j} = (j_1,\dots,j_T)\in\mathcal{N}^T$, define its number of switches by 
    \[ 
    S_T(\mathbf j) := \sum_{t=1}^{T-1}\mathds{1}_{\{j_t\neq j_{t+1}\}}, 
    \]
where $\mathds{1}_{A}$ denotes the indicator function, equal to $1$ if the event $A$ occurs and $0$ otherwise.
For a switching budget $K \in \{0,1,\dots, T-1\}$, define 
    $
    \mathcal J_{T,K} := \left\{ \mathbf j\in\mathcal{N}^T: S_T(\mathbf j)\leq K \right\}, 
    $
which represents the \emph{comparator class}, consisting of all expert sequences that switch at most $K$ times over the horizon $T$.
The switching budget $K$ is prescribed for each horizon $T$ and may depend on $T$. 
We use $K$ in finite-horizon statements and write $K(T)$ when considering the limit $T\to\infty$.
In asymptotic statements, ``fixed $K$'' means that the budget is independent of $T$.

\begin{definition}
    For a given horizon $T\geq 1$ and a switching budget $K\in\{0,1,\ldots,T-1\}$, the \emph{tracking regret} with switching budget $K$ is defined as
    \[
    R_{\rm track}(T,K)
    :=
    \sum_{t=1}^T \ell(\widehat{\mathbf w}(t),\mathbf{y}(t))
    -
    \min_{\mathbf j\in\mathcal J_{T,K}}
    \sum_{t=1}^T \ell(\mathbf w_{j_t}(t),\mathbf{y}(t)).
    \]
\end{definition}

Each sequence $\mathbf j\in\mathcal J_{T,K}$ partitions
$\{1,\ldots,T\}$ into at most $K+1$ consecutive blocks, within each of
which the comparator follows one fixed expert. 
The minimization over $\mathcal{J}_{T,K}$ selects the
best such expert sequence in hindsight; see \cite{cesa2006prediction}.

\begin{remark}[Static Regret as a Special Case]
When $K=0$, the comparator must follow one fixed expert throughout the
horizon; i.e., $\mathcal{J}_{T,0} = \{\mathbf{j}=(j,j,\dots,j) \in \mathcal{N}^T: j \in \mathcal{N}\}$. 
Hence, static regret is the special~case
\[
R_{\rm static}(T)
:=
R_{\rm track}(T,0)
=
\sum_{t=1}^T
\ell(\widehat{\mathbf w}(t),\mathbf y(t))
-
\min_{j\in\mathcal N}
\sum_{t=1}^T
\ell(\mathbf w_j(t),\mathbf y(t)).
\]
\end{remark}

\subsection{Regret Minimization Problem}
Our goal is to dynamically aggregate window-size experts and control the cost-sensitive tracking regret relative to the best expert sequence under a prescribed switching budget.
We explicitly account for the turnover of the aggregated portfolio in the regret criterion. 
Assume further that $a\leq\ell(\mathbf w,\mathbf y)\leq b$ for all $(\mathbf w,\mathbf y)\in\mathcal W\times\mathcal Y$, where $a<b$ are fixed finite constants.
Let $c(t)\geq0$ denote the \emph{transaction cost rate}. 
We now formalize our main problem.

\begin{problem}[Cost-Sensitive Regret Minimization]
\label{problem: cost-sensitive regret minimization}
Define the \emph{cost-sensitive loss of expert $j$ at time $t$} as their prediction loss plus their incurred transaction cost:
$$
\tilde{\ell}_j(t) 
:= 
\ell(\mathbf{w}_j(t), \mathbf{y}(t)) + c(t)\Vert\mathbf{w}_j(t) - \mathbf{w}_j(t-1)\Vert_1.
$$
Given a time horizon $T\ge1$ and a switching budget $K\in\{0,\ldots,T-1\}$, our goal is to select $\mathbf q(t)\in\Delta_n$ before $\mathbf y(t)$ is revealed and establish a worst-case upper bound on the resulting \emph{cost-sensitive tracking regret}:
\begin{align*}
R_{\rm track}^{c}(T,K)
:= &
\sum_{t=1}^T
\left[
    \ell(\widehat{\mathbf w}(t),\mathbf y(t))
    +c(t)\|\widehat{\mathbf w}(t)-\widehat{\mathbf w}(t-1)\|_1
\right]
-
\min_{\mathbf j\in\mathcal J_{T,K}}
\sum_{t=1}^T\tilde{\ell}_{j_t}(t),
\end{align*}
where $\widehat{\mathbf w}(0)=\mathbf w_0$ and $\mathcal J_{T,K}$ consists of all expert sequences that switch at most $K$ times over the horizon $T$.
\end{problem}

\paragraph{Hannan Consistency for Cost-Sensitive Tracking Regret.}
Beyond finite-horizon regret bounds, we seek asymptotic no regret. 
We formalize this objective through a cost-sensitive tracking analog of Hannan consistency~\cite{hannan1957approximation}.

\begin{definition}[Cost-Sensitive Hannan Consistency]
Given a prescribed switching-budget sequence
$K(T)\in\{0,1,\ldots,T-1\}$, we say that an algorithm is \emph{Hannan consistent for cost-sensitive tracking regret} with respect to $\{\mathcal J_{T,K(T)}\}_{T\ge1}$ if, for every outcome sequence
$\{\mathbf y(t)\}_{t\ge1}\subseteq\mathcal Y$
and every transaction-cost sequence
$\{c(t)\}_{t\ge1}\subseteq[0,1]$,
\[
    \limsup_{T\to\infty}
    \frac{R_{\rm track}^{c}(T,K(T))}{T}
    \le0.
\]
\end{definition}

\begin{remark}[Role of the Switching Budget]
    The switching budget restricts the comparator sequences, not the evolution of market returns.
    For asymptotic analysis, we consider prescribed budgets satisfying
    \begin{equation}\label{eq: switch time limit}
        K(T)=o(T).
    \end{equation}
    No-regret guarantees for these budgets depend on the loss assumptions and the algorithm's parameter choices.
\end{remark}

\section{Theoretical Results}

\subsection{Cost-Sensitive Regret Bounds}

Using the cost-sensitive expert losses defined in
Problem~\ref{problem: cost-sensitive regret minimization},
we establish the following regret decomposition
for any aggregation rule.
For notational convenience, set $\mathbf q(0):=\mathbf q(1)$.

\begin{theorem}[Cost-Sensitive Tracking Regret Bound]
\label{thm: regret_bound_tc}
    For any time horizon $T\ge1$, let $c(t) \ge 0$ be a transaction cost rate at time $t$. 
    Consider any algorithm $\mathcal{A}$ that generates aggregation weights $\mathbf{q}(t) \in \Delta_n$. 
    For any minimizing benchmark sequence of experts in $\mathcal{J}_{T,K}$, the cost-sensitive tracking regret $R_{\rm{track}}^c(T, K) $~satisfies
    \begin{align} \label{cost-sensitive regret bound}
        R_{\rm{track}}^c(T, K)  
        \leq 
        \underbrace{\sum_{t=1}^T \left[ 
            \left( \sum_{j\in\mathcal{N}} q_j(t) \tilde{\ell}_j(t) \right) - \tilde{\ell}_{j_t}(t) \right]}_{\text{Expert-loss regret bound}} 
        + 
        \underbrace{\sum_{t=1}^T c(t)\| \mathbf{q}(t) - \mathbf{q}(t-1)\|_1}_{\text{Additional turnover cost regret bound}}.
    \end{align}
The coefficient $1$ of the additional turnover cost bound is sharp; i.e., it cannot be reduced uniformly over all admissible expert portfolios and aggregation rules.
\end{theorem}

\begin{proof}
Let $\widehat{\mathbf{w}}(t) = \sum_{j\in\mathcal{N}} q_j(t) \mathbf{w}_j(t)$ denote the aggregate portfolio. Define the combined loss of expert $j$ at time $t$ as their prediction loss plus their incurred transaction cost at step $t$:
    $$
    \tilde{\ell}_j(t) = \ell(\mathbf{w}_j(t), \mathbf{y}(t)) + c(t)\Vert\mathbf{w}_j(t) - \mathbf{w}_j(t-1)\Vert_1.
    $$
Because the original loss is bounded in $[a, b]$ and the maximum $l_1$ distance between any two vectors on the probability simplex is $2$, the combined loss at time $t$ lies in the interval $\tilde{\ell}_j(t) \in [a, b+2c(t)]$. 
The cost-sensitive tracking regret against a minimizing benchmark sequence of optimal experts $j_1, \dots, j_T$ with at most~$K$ shifts is defined as:
\begin{align}
\label{eq: cost-sensitive tracking regret}
    R_{\rm{track}}^c(T, K)  = \sum_{t=1}^T \left[ \ell(\widehat{\mathbf{w}}(t), \mathbf{y}(t)) + c(t)\Vert\widehat{\mathbf{w}}(t) - \widehat{\mathbf{w}}(t-1)\Vert_1 \right] - \sum_{t=1}^T \tilde{\ell}_{j_t}(t).
\end{align}
    
Since the loss function $\ell$ is convex in its first argument, Jensen's inequality guarantees that:
\begin{align}
\label{ineq: jensen}
    \ell(\widehat{\mathbf{w}}(t), \mathbf{y}(t)) 
    = \ell\left(\sum_{j\in\mathcal{N}} q_j(t) \mathbf{w}_j(t), \mathbf{y}(t)\right) 
    \le \sum_{j\in\mathcal{N}} q_j(t) \ell(\mathbf{w}_j(t), \mathbf{y}(t)).
\end{align}

On the other hand, observe that $c(t)\Vert\widehat{\mathbf{w}}(t) - \widehat{\mathbf{w}}(t-1)\Vert_1 $ in \eqref{eq: cost-sensitive tracking regret}
can be written as    
\begin{align*}
    c(t)\Vert\widehat{\mathbf{w}}(t) - \widehat{\mathbf{w}}(t-1)\Vert_1 
    &= c(t) \left\| \sum_{j\in\mathcal{N}} q_j(t) \mathbf{w}_j(t) - \sum_{j\in\mathcal{N}} q_j(t-1) \mathbf{w}_j(t-1) \right\|_1 .
\end{align*}
By adding/subtracting $\sum q_j(t)\mathbf{w}_j(t-1)$ yields:
\begin{align}
    & c(t)\Vert\widehat{\mathbf{w}}(t) - \widehat{\mathbf{w}}(t-1)\Vert_1 \notag\\
    &= c(t) \left\| \sum_{j\in\mathcal{N}} q_j(t) \mathbf{w}_j(t) 
    -\sum q_j(t)\mathbf{w}_j(t-1)
    +\sum q_j(t)\mathbf{w}_j(t-1)
    - \sum_{j\in\mathcal{N}} q_j(t-1) \mathbf{w}_j(t-1) \right\|_1 \notag\\
    &= c(t) \left\| \sum_{j\in\mathcal{N}} q_j(t) \left( \mathbf{w}_j(t) - \mathbf{w}_j(t-1) \right) + \sum_{j\in\mathcal{N}} \left( q_j(t) - q_j(t-1) \right) \mathbf{w}_j(t-1) \right\|_1 \notag\\
    &\le c(t)\sum_{j\in\mathcal{N}} q_j(t) \Vert\mathbf{w}_j(t) - \mathbf{w}_j(t-1)\Vert_1 + c(t)\sum_{j\in\mathcal{N}} \vert q_j(t) - q_j(t-1)\vert \Vert\mathbf{w}_j(t-1)\Vert_1. \notag \\
    &= 
    c(t)\sum_{j\in\mathcal{N}} q_j(t) \Vert\mathbf{w}_j(t) - \mathbf{w}_j(t-1)\Vert_1 
    + c(t)\| \mathbf{q}(t) - \mathbf{q}(t-1)\|_1,\label{ineq: c weight variation}
\end{align}
where the second-to-last inequality follows from the triangle inequality and $q_j(t) \in [0,1]$.

Because each expert's portfolio $\mathbf{w}_j(t-1) \in \mathcal{W}$, which lies on the probability simplex, its $l_1$-norm is exactly~$1$. 
Therefore, combining the \eqref{ineq: jensen} and the \eqref{ineq: c weight variation} allows us to reconstruct the expected combined loss $\tilde{\ell}_j(t)$:
    \begin{align}
    \label{ineq: aux_3}
        \ell(\widehat{\mathbf{w}}(t), \mathbf{y}(t)) + c(t)\Vert\widehat{\mathbf{w}}(t) - \widehat{\mathbf{w}}(t-1)\Vert_1 \le \sum_{j\in\mathcal{N}} q_j(t) \tilde{\ell}_j(t) + c(t)\| \mathbf{q}(t) - \mathbf{q}(t-1)\|_1.
    \end{align}
After substituting \eqref{ineq: aux_3} into \eqref{eq: cost-sensitive tracking regret} and rearranging, the statement is proved.

To establish sharpness, consider two assets and two experts,
relabeled as $1$ and $2$ for this example, with $T=2$,
$\ell\equiv0$, and $c(1)=0<c(2)$.
Let the two experts hold $\mathbf e_1$ and $\mathbf e_2$,
respectively, at both periods; i.e., $\mathbf w_1(t) = \mathbf e_1$ and $\mathbf w_2(t) = \mathbf e_2$ for all $t=1,2$. 
Take
$\mathbf q(1)=(1,0)$ and $\mathbf q(2)=(0,1)$
Then 
$\widehat{\mathbf{w}}(1) = \sum_{j=1}^2 q_j(1) \mathbf w_j(1) = \mathbf e_1$ 
and
$\widehat{\mathbf{w}}(2) = \sum_{j=1}^2 q_j(2) \mathbf w_j(2) =\mathbf e_2$. 
Hence, for $j=1,2,$
\begin{align} \label{eq: cost-sensitive losses for T_2 cases}
    \tilde\ell_j(1)
    =0+c(1)\|\mathbf e_j-\mathbf w_0\|_1=0
    \quad \text{ and } \quad
    \tilde\ell_j(2)
    =0+c(2)\|\mathbf e_j-\mathbf e_j\|_1=0. 
\end{align}
Hence, any admissible benchmark sequence has zero cumulative expert loss:
$\min_{\mathbf j\in\mathcal J_{2,K}}
\sum_{t=1}^2\tilde{\ell}_{j_t}(t) = 0.$
The corresponding cost-sensitive tracking regret
\begin{align*}
    R_{\rm track}^{c}(2,K)
    &= 
    \underbrace{c(1)\|\widehat{\mathbf w}(1)-\widehat{\mathbf w}(0)\|_1}_{=0}
    +
    c(2)\|\widehat{\mathbf w}(2)-\widehat{\mathbf w}(1)\|_1 - 0\\
    &= 
    0+ c(2)\| {\mathbf e}_2-{\mathbf e}_1\|_1\\
    &=2c(2),
\end{align*}
On the other hand, applying the setting above to the derived regret bound~\eqref{cost-sensitive regret bound} yields
\begin{align*} 
    &\sum_{t=1}^2 \left( \sum_{j=1}^2 q_j(t) \tilde{\ell}_j(t) - \tilde{\ell}_{j_t}(t) \right)
    + 
    \sum_{t=1}^2 c(t)\| \mathbf{q}(t) - \mathbf{q}(t-1)\|_1 \\
    &=
   \left(  \tilde{\ell}_1(1)  - \tilde{\ell}_{j_1}(1) \right)
     +
    \left(  \tilde{\ell}_2(2) - \tilde{\ell}_{j_2}(2) \right)
    + c(2)\| \mathbf{q}(2) - \mathbf{q}(1)\|_1\\
    & = 0 + 0 + 2c(2)
\end{align*}
where the last equality holds by~\eqref{eq: cost-sensitive losses for T_2 cases}.
We see hence that the cost-sensitive tracking regret attains the derived bound, which completes the proof.
\end{proof}

\subsection{Fixed Share for Window Selection}
We now specialize the aggregation rule to the Fixed Share algorithm~\cite{herbster1998tracking}, with aggregation weights updated from the turnover-inclusive expert losses $\tilde{\ell}_j(t)$. 
Initialize $q_j(1)=1/n$ for every~$j \in \mathcal{N}$. 
After observing $\mathbf{y}(t)$, the algorithm updates the aggregation weights for stage~$t+1$ in two steps:

\emph{Step 1: Exponential weighting.}
For expert $j$ at time $t$, compute the intermediate weight $q_j^{\mathrm{exp}}(t)$ using the rule:
    \begin{align} \label{eq: intermediate weight}
        q_j^{\mathrm{exp}}(t)
        =
        \frac{q_j(t)e^{-\eta\tilde{\ell}_j(t)}}
        {\sum_{k\in\mathcal N}q_k(t)e^{-\eta\tilde{\ell}_k(t)}},
    \end{align}
where $\eta>0$ is the learning rate.\footnote{
    The learning rate $\eta$ controls the trade-off between exploiting historically strong experts and adapting to changes in expert performance.
    A higher $\eta$ means faster adaptation.  
    Conversely, a smaller $\eta$ leads to more stable weights that change slowly.
    }

\emph{Step 2: Sharing.}
For a mixing parameter $\alpha \in [0,1]$, retain a $(1-\alpha)$ fraction of the intermediate weight and redistribute the remaining $\alpha$ fraction across  all $n$ experts uniformly:
    \begin{align} \label{eq: fixed share standard update}
        q_j(t+1) 
        &:= (1-\alpha) q_j^{\mathrm{exp}}(t) + \frac{\alpha}{n} \sum_{k\in\mathcal{N}} q_k^{\mathrm{exp}}(t) \notag \\
        &= (1-\alpha) q_j^{\mathrm{exp}}(t) + \frac{\alpha}{n},
    \end{align}
    where the last equality follows from $\sum_{k\in\mathcal{N}}q_k^{\mathrm{exp}}(t)=1$.

\begin{remark}
    The Hedge Algorithm \cite{littlestone1994information} is the special case of Fixed Share with~$\alpha=0$, for which $q_j(t+1)=q_j^{\mathrm{exp}}(t)$, exactly the~\eqref{eq: intermediate weight}.
    Hedge competes with the best fixed expert in hindsight. 
    For $\alpha >0$, the mixing step in Fixed Share guarantees $q_j(t+1) \geq \alpha/n$, ensuring a positive weight for every~expert. 
\end{remark}

\subsection{Fixed Share Regret Bounds}
To apply Theorem~\ref{thm: regret_bound_tc} to Fixed Share, we first bound the variation of its aggregation weights.

\begin{lemma}[One-Step Weight Variation] 
\label{lem: weight_variation}
    Let $\mathbf q(t) = (q_j(t))_{j\in\mathcal N}$ be generated by \eqref{eq: intermediate weight} and \eqref{eq: fixed share standard update} with learning rate $\eta>0$ and parameter $\alpha\in[0,1]$.
    Then, for every $t\ge2$, we have
    \[
    \|\mathbf q(t)-\mathbf q(t-1)\|_1
    \leq
    \eta 
    \left(
        \max_{j\in\mathcal{N}} \tilde{\ell}_j(t-1) 
        -
        \min_{j\in\mathcal{N}} \tilde{\ell}_j(t-1)
    \right)
    +2\alpha\left(1-\frac{1}{n}\right).
    \]
\end{lemma}
\begin{proof}
    The proof is given in Appendix~\ref{appendix: technical proofs}.
\end{proof}

Combining Theorem~\ref{thm: regret_bound_tc} and Lemma~\ref{lem: weight_variation} with the standard Fixed Share regret bound for the losses
$\tilde\ell_j(t)$ yields the following result.

\begin{corollary}[Cost-Sensitive Tracking Regret Bound for Fixed Share]
\label{thm: regret_bound_tc_adaptive} 
Assume Theorem \ref{thm: regret_bound_tc} holds.
Consider the Fixed Share algorithm with mixing parameter $\alpha \in [0, 1)$, the cost-sensitive tracking regret $R_{\rm{track}}^c(T, K) $ satisfies
\begin{align*}
    R_{\rm{track}}^c(T, K) 
    &\le
    \frac{\mathcal{C}(K, \alpha, n, T)}{\eta} + \frac{\eta}{8}\sum_{t=1}^T (b-a+2c(t))^2\\
    &
    \qquad 
    + 
    \eta\sum_{t=2}^T c(t)(b-a+2c(t-1))
    +
    2\alpha\left(1-\frac1n\right)\sum_{t=1}^T c(t), 
\end{align*}
where the term $\mathcal{C}(K, \alpha, n, T)$ is given by 
\begin{align*}
    \mathcal{C}(K, \alpha, n, T) :=
    \begin{cases}
        \log n & \text{if } \alpha=0 \text{ and } K=0 \\
        \infty & \text{if } \alpha=0 \text{ and } K>0 \\
        (K+1) \log n + K \log\frac{1}{\alpha} + (T-K-1)\log\frac{1}{1-\alpha} & \text{otherwise} .
    \end{cases}
\end{align*}
\end{corollary}
\begin{proof}
    The proof is given in Appendix~\ref{appendix: technical proofs}.
\end{proof}

\subsection{Parameter Selection and Asymptotic No-Regret}

We can further optimize the cost-sensitive tracking regret bounds obtained in Corollary~\ref{thm: regret_bound_tc_adaptive} when the transaction costs $c(t)$ are known for the entire horizon $T$. 

\begin{proposition}[Optimal Regret Bound]
\label{prop: optimized_bounds}
    Assume the conditions of Corollary~\ref{thm: regret_bound_tc_adaptive} hold, with $T\ge2$ and $K\in\{0,\ldots,T-2\}$. 
    Define 
$A_T :=
\frac{1}{8}\sum_{t=1}^T (b-a+2c(t))^2
+\sum_{t=2}^T c(t)(b-a+2c(t-1))$
, and $C_T := \sum_{t=1}^T c(t)$. 
    Let $H(x)=-x\log{x} -(1-x)\log(1-x)$ be the binary entropy function defined for $x \in [0, 1]$ with the convention $0\log 0:=0$.
    The cost-sensitive tracking regret $R_{\rm{track}}^c(T, K)$ is optimally bounded as follows:
    
    \textit{Case 1: For $C_T=0$.} 
    With the optimal mixing parameter $\alpha^*=\frac{K}{T-1} \in [0, 1)$ and the optimal learning rate $\eta$ is chosen as:
    \begin{align*} 
        \eta^* = \frac{1}{b-a}\sqrt{\frac{8}{T}\left((K+1)\log n + (T-1)H\left(\frac{K}{T-1}\right)\right)},
    \end{align*} 
    the resulting bound is
   $$ 
    R_{\rm{track}}^c(T, K) \leq (b-a)\sqrt{\frac{T}{2}\left((K+1)\log n + (T-1)H\left(\frac{K}{T-1}\right)\right)}.
    $$

\textit{Case 2: For $C_T > 0$ and $K=0$.} 
    With $\alpha^*=0$ and $\eta^* = \sqrt{\frac{\log n}{A_T}}$, the bound is
    $$
    R_{\rm{track}}^c(T, 0) \leq 2 \sqrt{A_T \log n}.
    $$

\textit{Case 3: For $C_T > 0$ and $K \geq 1$.}
    Let $\alpha^* \in \left(0, \frac{K}{T-1}\right)$ be the unique solution to the equation
    $$
    \frac{K - (T-1)\alpha}{\alpha(1-\alpha)} = 2\left(1-\frac{1}{n}\right) C_T \sqrt{\frac{\mathcal{C}(K, \alpha, n, T)}{A_T}}.
    $$
    For $\eta^* = \sqrt{\frac{\mathcal{C}(K, \alpha^*, n, T)}{A_T}}$, the bound is
    $$
    R_{\rm{track}}^c(T, K) \leq 2 \sqrt{ A_T \cdot \mathcal{C}(K, \alpha^*, n, T) } + 2\left(1-\frac{1}{n}\right) \alpha^* C_T.
    $$
\end{proposition}
\begin{proof}
    The proof is given in Appendix~\ref{appendix: technical proofs}.
\end{proof}

\begin{remark}[Cost-Sensitive Static Regret Bound]
    Setting $K=0$ and $\alpha=0$ in Corollary \ref{thm: regret_bound_tc_adaptive} and Proposition~\ref{prop: optimized_bounds} yields the cost-sensitive static regret for the Hedge algorithm and its optimal learning rate. 
\end{remark}

\begin{corollary}[Asymptotic Cost-Sensitive No-Regret]
\label{cor: asymptotic_no_regret}
    Assume the conditions of
    Corollary~\ref{thm: regret_bound_tc_adaptive},
    with $n,a,b$ fixed and $c(t) \in [0,1]$ for all $t$.
    Let $K=K(T)\in\{0,\ldots,T-1\}$ satisfy $K(T)=o(T)$.
    For each sufficiently large $T$, choose the parameters
    as in Proposition~\ref{prop: optimized_bounds}.
    Then the cost-sensitive tracking regret satisfies
    \[
        \limsup_{T\to\infty}
        \frac{R_{\rm track}^c(T,K(T))}{T}\le0.
    \]
\end{corollary}
\begin{proof}
    The proof is given in Appendix~\ref{appendix: technical proofs}.
\end{proof}

\section{Simulation Results}
\label{section: simulation results}
We first examine adaptation to a prescribed regime shift in a synthetic market,
then summarize results on historical stock data.
Additional synthetic experiments, comparisons with online portfolio selection
algorithms, and the full empirical studies appear in
Appendix~\ref{appendix: additional experiments}.

\subsection{Synthetic Data}
\paragraph{The Setup.}
    To illustrate our two-level framework, we first construct a synthetic environment that includes a prescribed regime shift. 
    We simulate daily prices for an artificial market of 32 stocks (\texttt{SYN\_1} to \texttt{SYN\_32}) over 2,000 business days, beginning January 1, 2020. 
    For this portfolio setting, we define our reference expert set as $\mathcal{N}= \{5, 10, 21, 63\}$.
    The transaction cost rate is set to $c=0.001$ (10 bps). 

    \emph{Data Generating Process.}
    Let $\bm{\mu}\in\mathbb{R}^{32}$ be the baseline daily expected return, $\alpha\in(0, 1)$ control the ratio of signal, and $\mathbf{\epsilon}(t)$ is s noise vector where $\mathbf{\epsilon}(t) \sim \mathcal{N}(0, \sigma^2)$.
    Let $\mathbf{y}(t)\in\mathbb{R}^{32}$ denote the return vector of the stocks at time $t$, which is defined as:
    $$
    \mathbf{y}(t) = \bm{\mu} + \alpha \left( \frac{1}{j_t} \sum_{k=1}^{j_t} \mathbf{y}({t-k}) - \bm{\mu} \right) + \mathbf{\epsilon}(t)
    $$
    where $j_t\in\mathcal{N}$ represents the target window size.  
    This construction induces a controlled change in the horizon that governs the conditional mean, allowing us to examine how the proposed cost-sensitive online aggregation framework responds to a known regime shift. 
    In our setup, the market undergoes a sharp regime shift precisely halfway through the dataset: for the first 1,000 days, the returns are driven by a $j_t = 63$, and $j_t = 5$ for the remaining days. 

    We then execute the Hedge and Fixed Share algorithm. 
    Setting the cost-sensitive loss function to negative PnL, Profit and Loss, with rate $c \ge 0$:
    $$
    \ell(\mathbf{f}_j(t), \mathbf{y}(t))= -\mathbf{f}_j(t)^\top \mathbf{y}(t) + c\Vert{}\mathbf{f}_j(t) - \mathbf{f}_j(t-1)\Vert{}_1.
    $$
    Since the number of experts $n$, the time horizon $T$, and the transaction cost rate $c$ are given, all the parameters needed can be optimized.

\paragraph{Performance Evaluation.}
Figure~\ref{fig: H+FS tracking expert syn} shows the cost-sensitive tracking
regret defined in Problem~\ref{problem: cost-sensitive regret minimization}.
Figure~\ref{fig: q_j syn 4} shows the aggregation weights $q_j(t)$,
with the thickness of each colored band representing the weight assigned
to one window-size expert.
In this experiment, Fixed Share reallocates its weights to the newly dominant
expert faster than Hedge.

\begin{figure}[htbp]
    \centering
    \includegraphics[width=.6\linewidth]{ 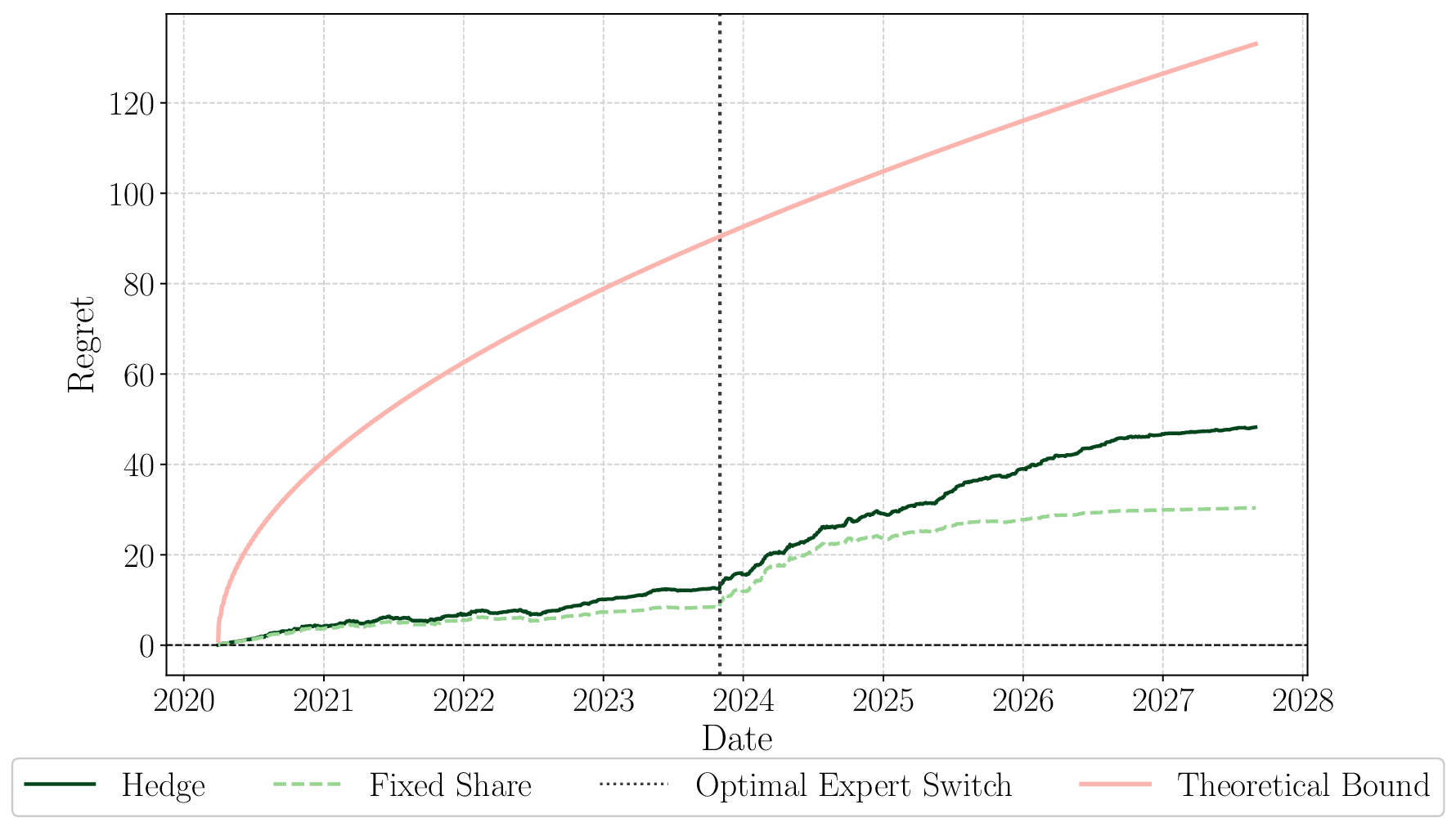}
    \caption{Cost-Sensitive Tracking Regret for Hedge and Fixed Share with Negative PnL Loss}
    \label{fig: H+FS tracking expert syn}
\end{figure}

\begin{figure}[htbp]
    \centering
    \begin{subfigure}[b]{0.45\textwidth}
        \centering
        \includegraphics[width=\textwidth]{ 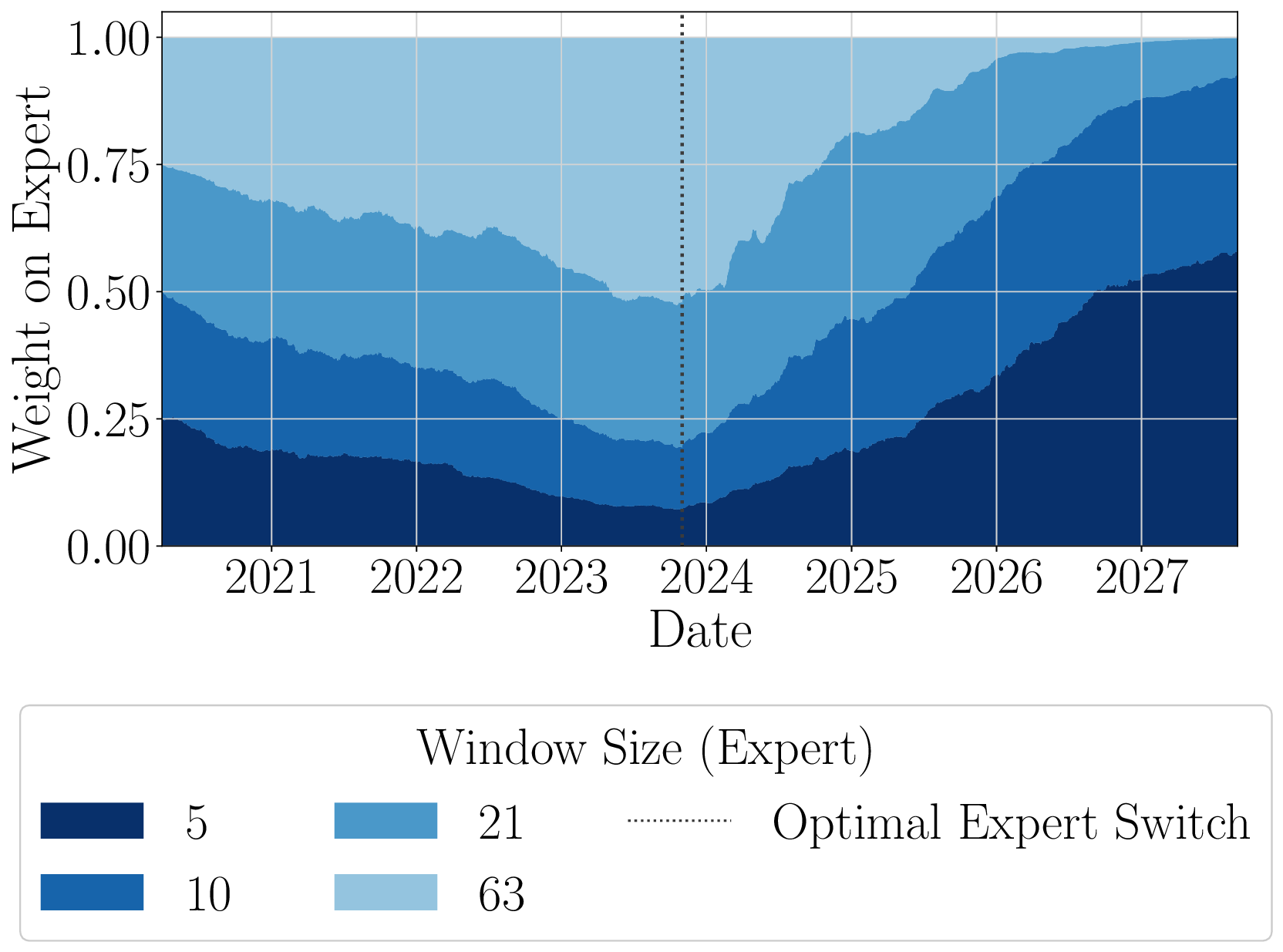}
        \caption{Hedge}
        \label{fig: H q_j syn stacked 4}
    \end{subfigure}
    \hfill
    \begin{subfigure}[b]{0.45\textwidth}
        \centering
        \includegraphics[width=\textwidth]{ 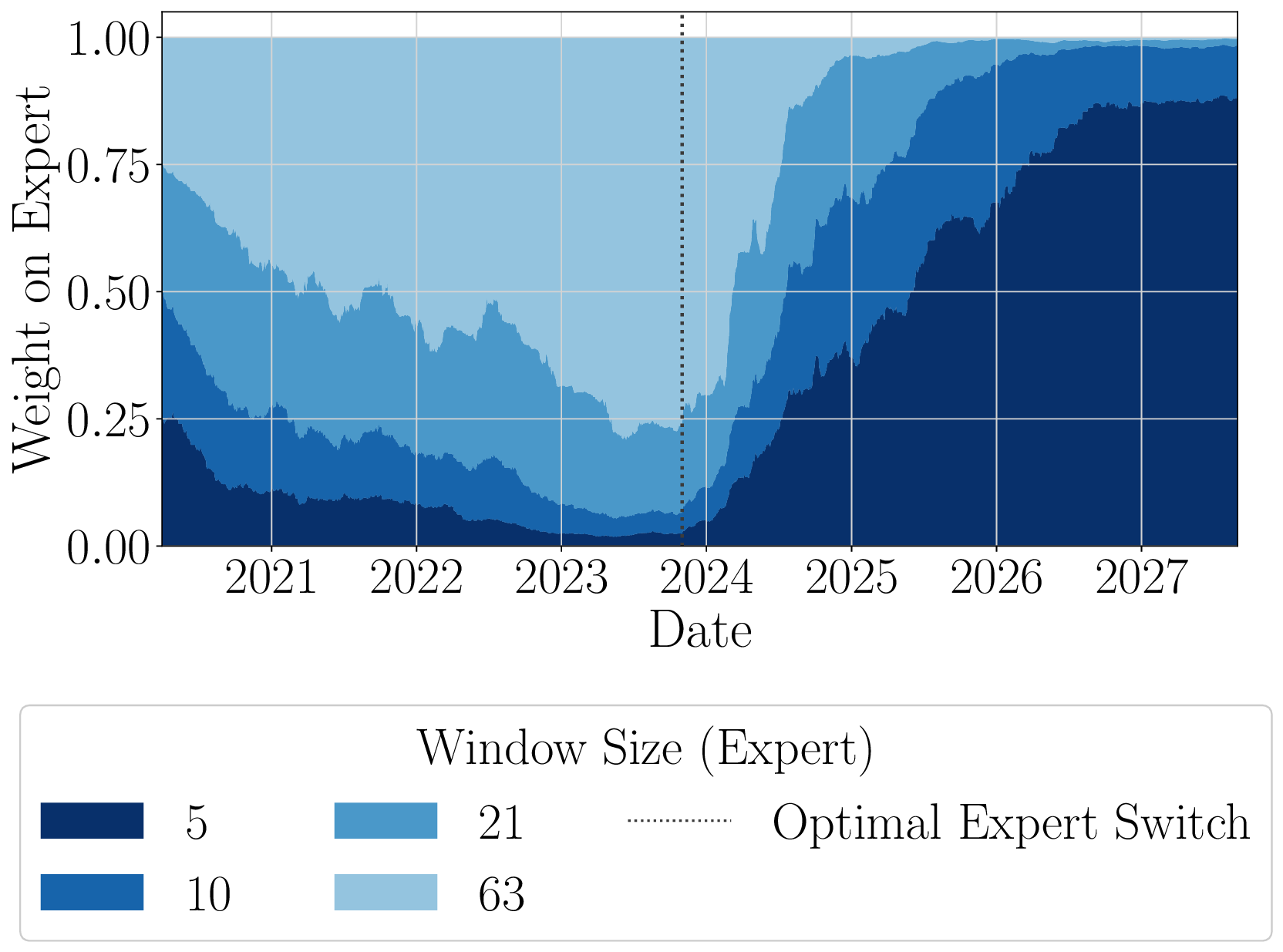}
        \caption{Fixed Share}
        \label{fig: FS q_j syn stacked 4}
    \end{subfigure}
   \caption{Evolution of the aggregation weights $q_j(t)$ assigned to the four window-size experts $j\in\{5,10,21,63\}$ under (a) Hedge and (b) Fixed Share.}    \label{fig: q_j syn 4}
\end{figure}

\subsection{Empirical Summary}
We also evaluate the framework on stock universes drawn from the DJIA and
S\&P 500, using $\mathcal N=\{5,10,21,30,41,50,63,126\}$ and $c=0.001$.
Table~\ref{table: empirical summary} summarizes cumulative returns and drawdowns.
In the DJIA sample, UP and EG achieve higher cumulative returns and smaller drawdown magnitude.
In the S\&P 500 sample, Hedge and Fixed Share achieve higher cumulative returns
than the listed baselines, but have larger drawdown magnitudes than UP and EG.
These results illustrate differing return--drawdown trade-offs across the two samples.
The account-value definition, comparison protocol, and full results are provided in
Appendix~\ref{appendix: additional experiments}. Note that 

\begin{table}[htbp]
    \centering
    \scriptsize
    \caption{Summary of the empirical studies with transaction cost rate $c=0.001$.
    DJIA: 2021/03/04--2024/02/25; S\&P 500: 2020/01/02--2026/08/01.
    Cumulative return and maximum drawdown (MDD) are reported in percent;
    MDD is signed, so values closer to zero indicate smaller drawdowns.}
    \label{table: empirical summary}
    \begin{tabular}{l r r r r}
        \toprule
        & \multicolumn{2}{c}{DJIA} & \multicolumn{2}{c}{S\&P 500} \\
        \cmidrule(lr){2-3}\cmidrule(lr){4-5}
        Algorithm & Return & MDD & Return & MDD \\
        \midrule
        Hedge & 21.1950 & -23.5123 & 567.2010 & -56.9601 \\
        Fixed Share & 17.7014 & -23.0759 & 461.4739 & -58.5937 \\
        UP & 48.7250 & -20.2783 & 205.3887 & -38.3346 \\
        EG & 48.2584 & -20.2783 & 205.7900 & -38.2862 \\
        PAMR & -88.2441 & -88.9822 & -52.0775 & -84.5912 \\
        CWMR & -88.5860 & -89.3036 & 205.5993 & -59.5749 \\
        OLMAR & -38.9035 & -61.0332 & -44.0283 & -81.6301 \\
        \bottomrule
    \end{tabular}
\end{table}

\section{Concluding Remarks}
We use Hedge and Fixed Share to dynamically aggregate window-size experts for portfolio selection while accounting for turnover costs. 
We derive a cost-sensitive tracking-regret bound for Fixed Share, with the static-regret bound for Hedge recovered as a special case.
The bound explicitly characterizes the dependence on the transaction-cost rate and guides the choice of the learning rate.

This paper focuses on tracking regret for expert sequences relative to expert sequences under a prescribed switching budget.
A complementary extension would be to embed the proposed cost-sensitive learner in a strongly adaptive meta-algorithm \cite{daniely2015strongly}, seeking regret guarantees uniformly over all subintervals. Such an extension requires a separate treatment of interval-specific parameterization and transaction costs and is left for future research.

\section*{AI Use Statement}
Generative AI tools (ChatGPT 5.6 sol and Gemini 3.8 Flash) were used to assist with language editing and presentation, and to provide feedback on mathematical derivations and proofs. The authors reviewed and independently verified all AI-assisted content. The authors take full responsibility for the final content of this work.

\bibliography{refs}
\bibliographystyle{iclr2027_conference}

\clearpage
\appendix

\section{Technical Proofs}
\label{appendix: technical proofs}

\begin{proof}[Proof of Lemma~\ref{lem: weight_variation}]
By adding and subtracting the intermediate weight vector $\mathbf{q}^{\mathrm{exp}}(t-1)$ and applying the triangle inequality, the total variation can be decomposed into the variation from the loss update and the variation from the sharing step:
\begin{align}
    \|\mathbf{q}(t) - \mathbf{q}(t-1)\|_1
    &=
    \|\mathbf{q}(t) - \mathbf{q}^{\mathrm{exp}}(t-1) + \mathbf{q}^{\mathrm{exp}}(t-1) - \mathbf{q}(t-1)\|_1  \notag \\
    &\leq
    \underbrace{\|\mathbf{q}^{\mathrm{exp}}(t-1) - \mathbf{q}(t-1)\|_1}_{\text{Loss Update Variation}} + \underbrace{\|\mathbf{q}(t) - \mathbf{q}^{\mathrm{exp}}(t-1)\|_1}_{\text{Sharing Variation}}.
    \label{ineq: total variation of q}
\end{align}
To bound the Loss Update Variation, define an auxiliary continuous function $\tilde{q}_j(s)$ parameterized by $s \in [0, \eta]$ that interpolates the weights during the exponential update:
    \begin{align} \label{eq: continuous auxiliary function}
    \tilde{q}_j(s) 
    := 
    \frac{q_j(t-1) e^{-s \tilde{\ell}_j(t-1)}}{Z(s)},
    \end{align}
where $Z(s) = \sum_{k\in\mathcal{N}} q_k(t-1) e^{-s \tilde{\ell}_k(t-1)}$. 
Note that $\tilde{q}_j(0) = q_j(t-1)$, $\tilde{q}_j(\eta) = q_j^{\mathrm{exp}}(t-1)$ and 
$\sum_{j\in\mathcal{N}}\tilde{q}_j(s)=1$.

For each expert $j$, by the Fundamental Theorem of Calculus, the absolute change is bounded by the integral of the absolute derivative:
    \begin{align}
        \vert q_j^{\mathrm{exp}}(t-1) - q_j(t-1) \vert 
        = \left\vert 
        \int_0^\eta 
        \frac{d}{ds} \tilde{q}_j(s) \, ds \right\vert 
        &\le 
        \int_0^\eta 
        \left\vert 
        \frac{d}{ds} \tilde{q}_j(s) 
        \right\vert \, ds. \label{eq: ftc}
    \end{align}
    
Using the logarithmic derivative identity, $\frac{d}{ds} \log \tilde{q}_j(s) 
= \frac{1}{\tilde{q}_j(s)}(\frac{d}{ds}\tilde{q}_j(s))$, then:
    \begin{align}
        \frac{1}{\tilde{q}_j(s)}\left(\frac{d}{ds}\tilde{q}_j(s)\right) 
        &= 
        \frac{d}{ds} \log \tilde{q}_j(s) \notag \\
        &= 
        -\tilde{\ell}_j(t-1) - \frac{d}{ds} \log Z(s).
    \label{eq: d tilde w}
    \end{align}
where the last equality follows from differentiating~\eqref {eq: continuous auxiliary function}.
    
Evaluating the derivative of the log-partition function $\log Z(s)$ yields the negative expected loss under the interpolated weights
$\{\tilde q_j(s)\}_{j\in\mathcal N}$; i.e.,
    \begin{align}
        \frac{d}{ds} \log Z(s)  
        &= \frac{1}{Z(s)} \sum_{k\in\mathcal{N}} q_k(t-1) (-\tilde{\ell}_k(t-1)) e^{-s \tilde{\ell}_k(t-1)} \notag \\ 
        &= -\sum_{k\in\mathcal{N}} \tilde{q}_k(s) \tilde{\ell}_k(t-1).     \label{eq: d log Z}
    \end{align}

    Substituting \eqref{eq: d log Z} back to \eqref{eq: d tilde w} and isolating the derivative gives:
    \begin{align}
        \frac{d}{ds} \tilde{q}_j(s) 
        = 
        \tilde{q}_j(s) \left( \sum_{k\in\mathcal{N}} \tilde{q}_k(s) \tilde{\ell}_k(t-1) - \tilde{\ell}_j(t-1) \right). \label{eq: auxiliary derivative}
    \end{align}
Applying \eqref{eq: auxiliary derivative} to \eqref{eq: ftc} yields 
    \begin{align}
        \vert q_j^{\mathrm{exp}}(t-1) - q_j(t-1) \vert 
        &\le \int_0^\eta 
        \left\vert 
        \tilde{q}_j(s)
        \left(
        \sum_{k\in\mathcal{N}} \tilde{q}_k(s) \tilde{\ell}_k(t-1) 
        - \tilde{\ell}_j(t-1) 
        \right)
        \right\vert \, ds\notag\\
        &\le \int_0^\eta 
        \tilde{q}_j(s)
        \left\vert 
        \sum_{k\in\mathcal{N}} \tilde{q}_k(s) \tilde{\ell}_k(t-1) 
        - \tilde{\ell}_j(t-1) 
        \right\vert \, ds.\label{eq: aux_2}
    \end{align}

Summing up over all $n$ experts yields:
\begin{align*}
    \| \mathbf{q}^{\mathrm{exp}}(t-1) - \mathbf{q}(t-1) \|_1 
    &\le 
    \sum_{j\in\mathcal{N}} \int_0^\eta \tilde{q}_j(s)\left(\left|\sum_{k\in\mathcal{N}} \tilde{q}_k(s) \tilde{\ell}_k(t-1) 
    -  \tilde{\ell}_j(t-1) \right| \right) \, ds \\
    &= 
     \int_0^\eta \sum_{j\in\mathcal{N}} \tilde{q}_j(s)
     \left(\left|\sum_{k\in\mathcal{N}} \tilde{q}_k(s) \tilde{\ell}_k(t-1) 
    -  \tilde{\ell}_j(t-1) \right| \right)  \, ds  \\
    &\le
    \int_0^\eta
    \sum_{j\in\mathcal N}\tilde q_j(s)
    \left(
    \max_{k\in\mathcal N}\tilde\ell_k(t-1)
    -\min_{k\in\mathcal N}\tilde\ell_k(t-1)
    \right)\,ds\\
    &=
       \eta 
        \left(
        \max_{k\in\mathcal N}\tilde\ell_k(t-1)
        -\min_{k\in\mathcal N}\tilde\ell_k(t-1)
        \right)
\end{align*}
where the last inequality follows because both $\sum_{k\in\mathcal N}\tilde q_k(s)\tilde\ell_k(t-1)$ and $\tilde\ell_j(t-1)$ lie in the interval
$[\min_{k\in\mathcal N}\tilde\ell_k(t-1),
  \max_{k\in\mathcal N}\tilde\ell_k(t-1)]$,
and the last equality uses $\sum_{j\in\mathcal N}\tilde q_j(s)=1$.

Next, we bound the Sharing Variation term in~\eqref{ineq: total variation of q} by substituting the update step $q_j(t) = (1-\alpha)q_j^{\mathrm{exp}}(t-1) + \frac{\alpha}{n}$:
    \begin{align*}
        \|\mathbf{q}(t) 
        - 
        \mathbf{q}^{\mathrm{exp}}(t-1)
        \|_1 
        &= 
        \left\| 
        (1-\alpha)\mathbf{q}^{\mathrm{exp}}(t-1) + 
        \frac{\alpha}{n}\mathbf{1} - \mathbf{q}^{\mathrm{exp}}(t-1) \right\|_1 \\
        &= 
        \left\| \frac{\alpha}{n}\mathbf{1} - \alpha \mathbf{q}^{\mathrm{exp}}(t-1) \right\|_1\\
        &=
        \alpha\left\| \frac{1}{n}\mathbf{1} -  \mathbf{q}^{\mathrm{exp}}(t-1) \right\|_1\\
        &\le 
        2\alpha\left(1-\frac{1}{n}\right),
    \end{align*}
where the last inequality follows from the convexity of the $\ell_1$-norm and the fact that $\|\frac{1}{n}\mathbf{1}-\mathbf e_j\|_1=2(1-\frac{1}{n})$ for every $j\in\mathcal N$. 
Here $\mathbf e_j\in\Delta_n$ assigns unit mass to expert $j$ and zero mass to every other expert.
The bound is sharp over the probability simplex, with equality at its vertices.
    
Combining the bounds for both components yields the final result:
$$
    \|\mathbf{q}(t) - \mathbf{q}(t-1)\|_1 
    \leq 
    \eta 
    \left(
    \max_{j\in\mathcal{N}} \tilde{\ell}_j(t-1) 
    -
    \min_{j\in\mathcal{N}} \tilde{\ell}_j(t-1)
    \right) 
    + 
    2\alpha\left(1-\frac{1}{n}\right),  
$$ 
which completes the proof.  \qedhere
\end{proof}

\begin{proof}[Proof of Corollary~\ref{thm: regret_bound_tc_adaptive}]
Recall Theorem \ref{thm: regret_bound_tc}, for the cost-sensitive tracking regret we have
    $$
    R_{\rm{track}}^c(T, K)  
    \leq 
    \underbrace{\sum_{t=1}^T \left[ 
        \left( \sum_{j\in\mathcal{N}} q_j(t) \tilde{\ell}_j(t) \right) - \tilde{\ell}_{j_t}(t) \right]}_{\text{Expert-loss regret bound}} 
    + 
    \underbrace{\sum_{t=1}^T c(t)\| \mathbf{q}(t) - \mathbf{q}(t-1)\|_1}_{\text{Additional turnover cost regret bound}}.
    $$
For the Regret Bound for $\mathcal{A}$, evaluated on losses in $[a, b+2c(t)]$, by \cite{cesa2006prediction}, standard analysis for the Fixed Share algorithm bounds the regret by:
    $$
    \text{Expert-loss regret bound }
    \le 
    \frac{\mathcal{C}(K, \alpha, n, T)}{\eta} + \frac{\eta}{8}\sum_{t=1}^T (b-a+2c(t))^2.
    $$
For the Regret Bound for Transaction Cost, using $\mathbf q(0)=\mathbf q(1)$
and Lemma~\ref{lem: weight_variation}, we obtain
\begin{align*}
\text{Additional Turnover cost regret bound}
&= \sum_{t=1}^T c(t)
   \|\mathbf q(t)-\mathbf q(t-1)\|_1\\
&= 0 + \sum_{t=2}^T c(t)
   \|\mathbf q(t)-\mathbf q(t-1)\|_1\\
&\le \sum_{t=2}^T c(t)\left[
   \eta\left(
   \max_{j\in\mathcal N}\tilde\ell_j(t-1)
   -\min_{j\in\mathcal N}\tilde\ell_j(t-1)
   \right)
   +2\alpha\left(1-\frac1n\right)\right]\\
&\le \eta\sum_{t=2}^T c(t)(b-a+2c(t-1))
   +2\alpha\left(1-\frac1n\right)\sum_{t=1}^T c(t).
\end{align*}
The last inequality follows from
$\tilde\ell_j(t-1)\in[a,b+2c(t-1)]$
and $c(t)\ge0$.

Combining both components yields the final regret bound. \qedhere
\end{proof}

\begin{proof}[Proof of Proposition~\ref{prop: optimized_bounds}]
By Corollary~\ref{thm: regret_bound_tc_adaptive}, the regret boundcan be expressed as a joint function of $\eta$ and $\alpha$:
    $$
    R(\eta, \alpha) = \frac{\mathcal{C}(K, \alpha, n, T)}{\eta} + \eta A_T + 2\alpha\left(1-\frac{1}{n}\right) C_T.
    $$
For $C_T = 0$. 
Since there are no transaction costs, the setting reduces to the standard Fixed Share algorithm; see \cite{herbster1998tracking}, \cite{cesa2006prediction}.

For $K=0$, if $\alpha>0$, the derivative with respect to $\alpha$ is 
    $$
    \frac{\partial R}{\partial \alpha} 
    = \frac{1}{\eta} \left( \frac{(T-1)\alpha - K}{\alpha(1-\alpha)} \right) + 2\left(1-\frac{1}{n}\right)C_T,
    $$
which is strictly positive for all $\alpha\in(0, 1)$ with $T>1$ and $\eta>0$.
The regret function $R$ is strictly increasing in $\alpha$. 
Therefore, it cannot have a minimum where the derivative is zero; the minimum must lie on the boundary constraint, i.e., $\alpha = 0$.
With $K=0$ and $\alpha = 0$, the complexity term simplifies to $\mathcal{C}(0, 0, n, T) = \log n$. 
The regret bound reduces to a function of $\eta$ alone:
    $$
    R(\eta, 0) = \frac{\log n}{\eta} + \eta A_T.
    $$
Taking the derivative with respect to $\eta$ and setting it to zero yields $\eta^* = \sqrt{\frac{\log n}{A_T}}$. Substituting $\eta^*$ back into $R(\eta, 0)$ gives the minimized bound $R_{\rm{track}}^c(T, 0) \le 2\sqrt{A_T \log n}$.

For the rest of the cases, $K>0$, $\alpha \in (0, 1)$ and $\eta > 0$, we optimize $R(\eta, \alpha)$ sequentially. For any fixed $\alpha \in (0, 1)$, the first and second partial derivatives with respect to $\eta$ are:
    $$
    \frac{\partial R}{\partial \eta} 
    = 
    -\frac{\mathcal{C}(K, \alpha, n, T)}{\eta^2} 
    + A_T
    , \quad 
    \frac{\partial^2 R}{\partial \eta^2} 
    = 
    \frac{2\mathcal{C}(K, \alpha, n, T)}{\eta^3}.
    $$
    Since $\mathcal{C}(K, \alpha, n, T) > 0$ and $\eta > 0$, the second partial derivative is strictly positive. Thus, $R(\eta, \alpha)$ is strictly convex in $\eta$, and setting $\frac{\partial R}{\partial \eta} = 0$ yields the unique conditional minimum:
    $$
    \eta^*(\alpha) = \sqrt{\frac{\mathcal{C}(K, \alpha, n, T)}{A_T}}.
    $$
Substituting $\eta^*(\alpha)$ back into the objective function yields the profile function $g(\alpha)$ depending solely on $\alpha$:
    \begin{align*}
        g(\alpha) 
        &:= R(\eta^*(\alpha), \alpha) \\
        &= \frac{\mathcal{C}(K, \alpha, n, T)}{\sqrt{\mathcal{C}(K, \alpha, n, T)/A_T}} + A_T\sqrt{\frac{\mathcal{C}(K, \alpha, n, T)}{A_T}} + 2\alpha\left(1-\frac{1}{n}\right) C_T \\
        &= 2\sqrt{A_T \cdot \mathcal{C}(K, \alpha, n, T)} + 2\alpha\left(1-\frac{1}{n}\right) C_T.
    \end{align*}
To find the minimum of $g(\alpha)$, we compute its first derivative:
    $$
    g'(\alpha) 
    = 
    \sqrt{\frac{A_T}{\mathcal{C}(K, \alpha, n, T)}} \left( -\frac{K}{\alpha} +\frac{T-K-1}{1-\alpha} \right)
    + 2\left(1-\frac{1}{n}\right)C_T .
    $$
    Setting $g'(\alpha) = 0$ yields the stationarity condition:
    $$
    \frac{K - (T-1)\alpha}{\alpha(1-\alpha)} 
    = 
    2C_T
    \left(1-\frac{1}{n}\right) 
    \sqrt{\frac{\mathcal{C}(K, \alpha, n, T)}{A_T}}.
    $$
To establish that the root $\alpha^*$ is the unique global minimum, we prove that $g(\alpha)$ is strictly convex on the domain $\alpha \in \left(0, \frac{K}{T-1}\right)$. Since $g(\alpha)$ is the sum of a linear term and $2\sqrt{A_T \cdot \mathcal{C}(K, \alpha, n, T)}$, it suffices to show that $\sqrt{\mathcal{C}}$ is strictly convex. The second derivative of $\sqrt{\mathcal{C}}$ is strictly positive if and only if $2\mathcal{C}\mathcal{C}'' > (\mathcal{C}')^2$.
    
The domain constraint $\alpha < \frac{K}{T-1}$. By subtracting $\alpha K$ at the both sides of $\alpha(T-1) < K$, we algebraically rearranges to $K(1-\alpha) > (T-K-1)\alpha$, which implies 
    \begin{align}
        \label{ineq: aux_ga_1}
        \frac{K}{\alpha} > \frac{T-1-K}{1-\alpha} > 0.
    \end{align}
The derivatives of $\mathcal{C}$ with respect to $\alpha$ can be expressed as $\mathcal{C}' = -\frac{K}{\alpha} + \frac{T-1-K}{1-\alpha}$.
Because \eqref{ineq: aux_ga_1}, we have 
    \begin{align*}
        (\mathcal{C}')^2 
        < (\frac{K}{\alpha})^2 
        = \frac{K^2}{\alpha^2}.
    \end{align*}
Furthermore, since $\frac{T-1-K}{(1-\alpha)^2} > 0$, we know 
    \begin{align}
    \label{ineq: aux_ga_3}
        \mathcal{C}'' 
        =
        \frac{K}{{\alpha}^2}+\frac{T-1-K}{(1-\alpha)^2}
        > 
        \frac{K}{\alpha^2}.
    \end{align} 
By definition, since other term are always positive in $\alpha\in(0, \frac{K}{T-1})$, we have 
    \begin{align}
    \label{ineq: aux_ga_4}
        \mathcal{C}(\alpha) 
        > (K+1)\log n.
    \end{align}
Combining \eqref{ineq: aux_ga_3} and \eqref{ineq: aux_ga_4} yields:
    $$ 
    2\mathcal{C}\mathcal{C}'' 
    >
    \frac{2K(K+1)\log n}{\alpha^2}. 
    $$
For any $n \ge 2$, we have $\log n \ge \log 2 > 0.5$, which guarantees the coefficient $2(K+1)\log n > K+1 > K$, yields
    $$ 
    2\mathcal{C}\mathcal{C}'' 
    > \frac{K^2}{\alpha^2} 
    > (\mathcal{C}')^2. 
    $$
Thus, $\sqrt{\mathcal C}$ and $g$ are strictly convex on
$(0,\alpha_0)$, where $\alpha_0:=K/(T-1)$.
Moreover,
\[
    \lim_{\alpha\downarrow0}g'(\alpha)=-\infty,
    \qquad
    g'(\alpha_0)=2\left(1-\frac1n\right)C_T>0.
\]
By continuity and strict monotonicity of $g'$ on
$(0,\alpha_0)$, there exists a unique root
$\alpha^*\in(0,\alpha_0)$.
For $\alpha\in[\alpha_0,1)$, we have
$\mathcal C'(\alpha)\ge0$, and hence $g'(\alpha)>0$.
Therefore, $\alpha^*$ is the unique global minimizer
of $g$ on $(0,1)$.
Evaluating $g(\alpha^*)$ completes the proof.\qedhere
\end{proof}

\begin{proof}[Proof of Corollary~\ref{cor: asymptotic_no_regret}]
Since $c(t)\in[0,1]$, we have $A_T=\mathcal{O}(T)$ and $C_T\le T$.
For sufficiently large $T$, the choice $\bar{\alpha}=K/(T-1)$ belongs to $[0,1)$.
By the definition of $\mathcal{C}$, using its first case
when $K=0$ and its last case when $K>0$, we obtain
\[
    B_T:=\mathcal{C}(K,\bar{\alpha},n,T)
    =(K+1)\log n
    +(T-1)H\left(\frac{K}{T-1}\right),
\]
where $H(0)=0$.
Since $K=o(T)$ and $H(x)\to0$ as $x\to0$,
we have $B_T=o(T)$.
The optimized upper bound is no larger than the bound obtained at $\bar{\alpha}$ and $\bar{\eta}=\sqrt{B_T/A_T}$. 
Hence,
\[
    \frac{R_{\rm track}^c(T,K)}{T}
    \le
    2\sqrt{\frac{A_T}{T}\frac{B_T}{T}}
    +2\left(1-\frac1n\right)
    \frac{K}{T-1}\frac{C_T}{T}.
\]
Both terms on the right converge to zero.
\end{proof}

\clearpage
\section{Additional Experiments}
\label{appendix: additional experiments}
This appendix provides the evaluation details and additional experiments
supporting Section~\ref{section: simulation results}.

\subsection{Account-Value Evaluation}
\label{appendix: account value}
After making decision $\widehat{\mathbf{w}}(t)\in\mathcal{W}$ at time $t$, the trader's account value evolves according to
\begin{align*}
    V(t+1)
    &=V(t)+V(t)\big(\widehat{\mathbf{w}}(t)^\top\mathbf{y}(t)\big) - c \|\widehat{\mathbf{w}}(t) - \widehat{\mathbf{w}}(t-1)\|_1 V(t)\\
    &=
    V(t)\big(
    1+\widehat{\mathbf{w}}(t)^\top\mathbf{y}(t)
    -c\left\|
    \widehat{\mathbf{w}}(t)-\widehat{\mathbf{w}}(t-1)
    \right\|_1
    \big)
\end{align*}    
where $V(1)>0$ and $\widehat{\mathbf w}(0)=\mathbf w_0$.
For this evaluation, we assume
$
1+\widehat{\mathbf w}(t)^\top\mathbf y(t)
-c\|\widehat{\mathbf w}(t)-\widehat{\mathbf w}(t-1)\|_1
>0
$
at every trading period.
This condition ensures positive account values.

\subsection{Additional Synthetic Experiments}
\label{appendix: synthetic variants}
We vary the synthetic setup in Section~\ref{section: simulation results}.
Figure~\ref{fig: q_j syn 4_2} illustrates the effects of altering the timing
of the regime shift, while Figure~\ref{fig: q_j syn} expands the expert set to
$\mathcal N=\{5,10,21,30,41,50,63,126\}$.

\begin{figure}[htbp]
    \centering
    \begin{subfigure}[b]{0.48\textwidth}
        \centering
        \includegraphics[width=\textwidth]{ 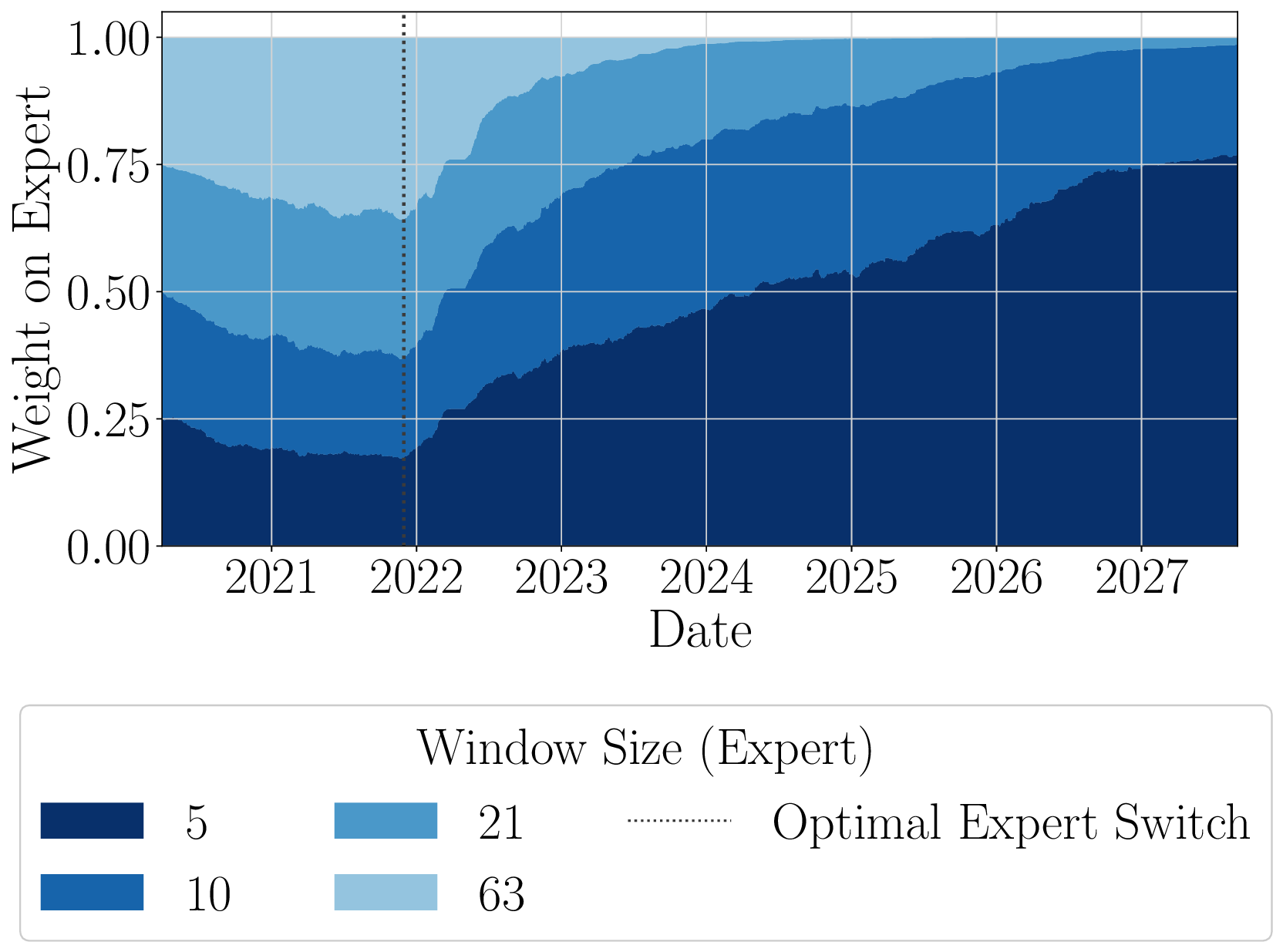}
        \caption{Hedge}
        \label{fig: H q_j syn stacked 4_2}
    \end{subfigure}
    \hfill
    \begin{subfigure}[b]{0.45\textwidth}
        \centering
        \includegraphics[width=\textwidth]{ 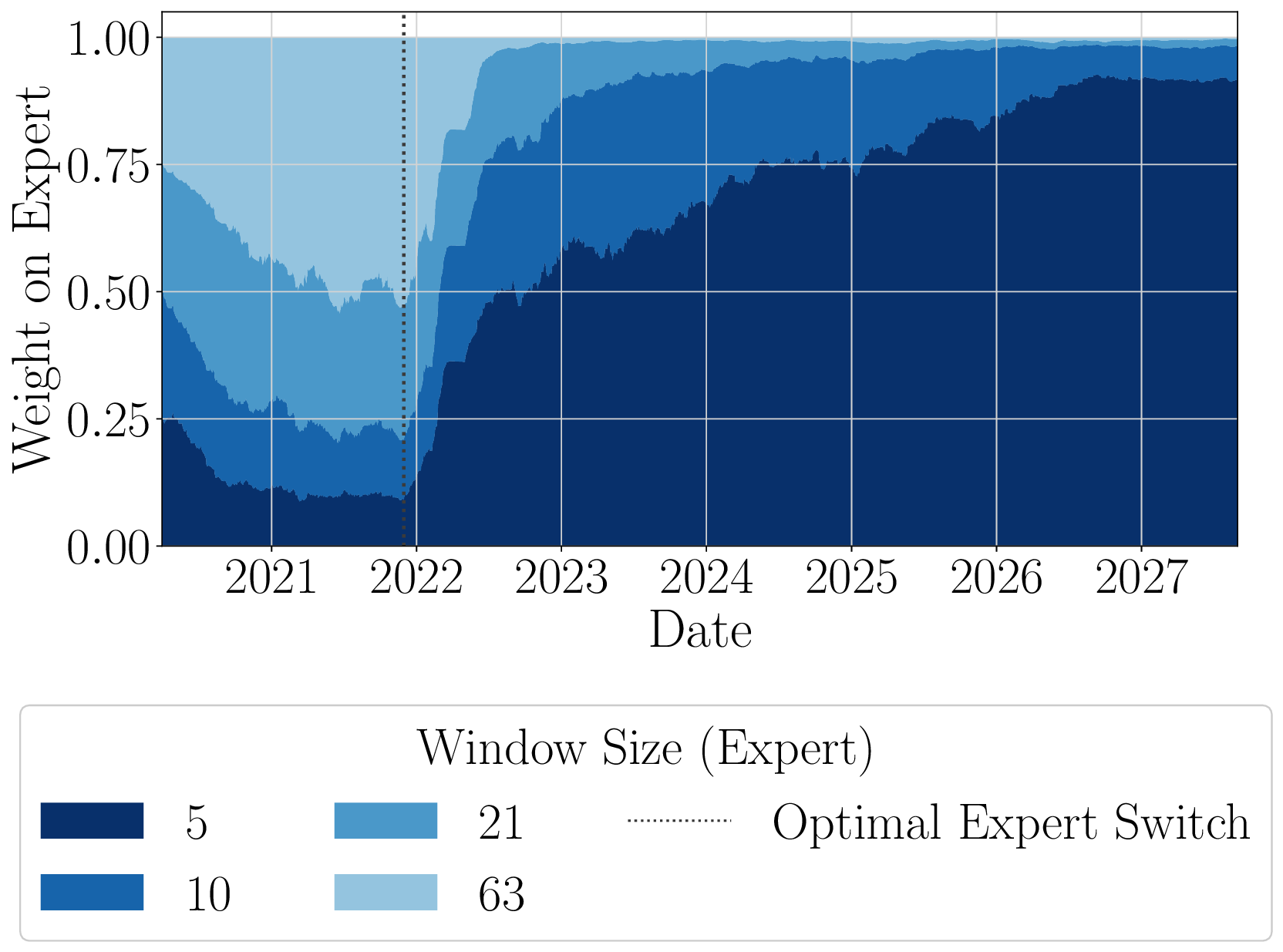}
        \caption{Fixed Share}
        \label{fig: FS q_j syn stacked 4_2}
    \end{subfigure}
    \caption{Evolution of the aggregation weights $q_j(t)$ on four Window Sizes for different switching timing.}
    \label{fig: q_j syn 4_2}
\end{figure}

\begin{figure}[htbp]
    \centering
    \begin{subfigure}[b]{0.45\textwidth}
        \centering
        \includegraphics[width=\textwidth]{ 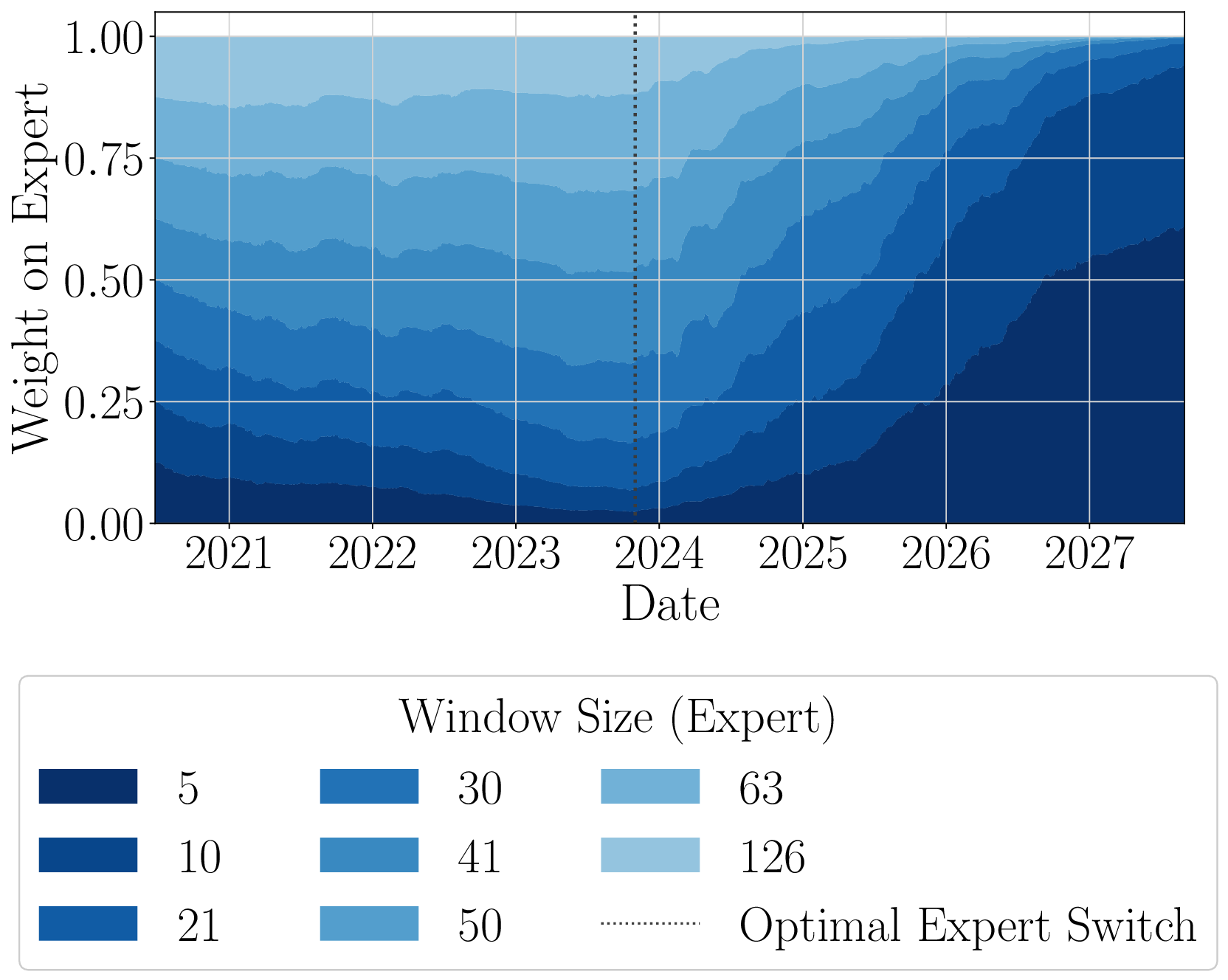}
        \caption{Hedge}
        \label{fig: H q_j syn stacked }
    \end{subfigure}
    \hfill
    \begin{subfigure}[b]{0.48\textwidth}
        \centering
        \includegraphics[width=\textwidth]{ 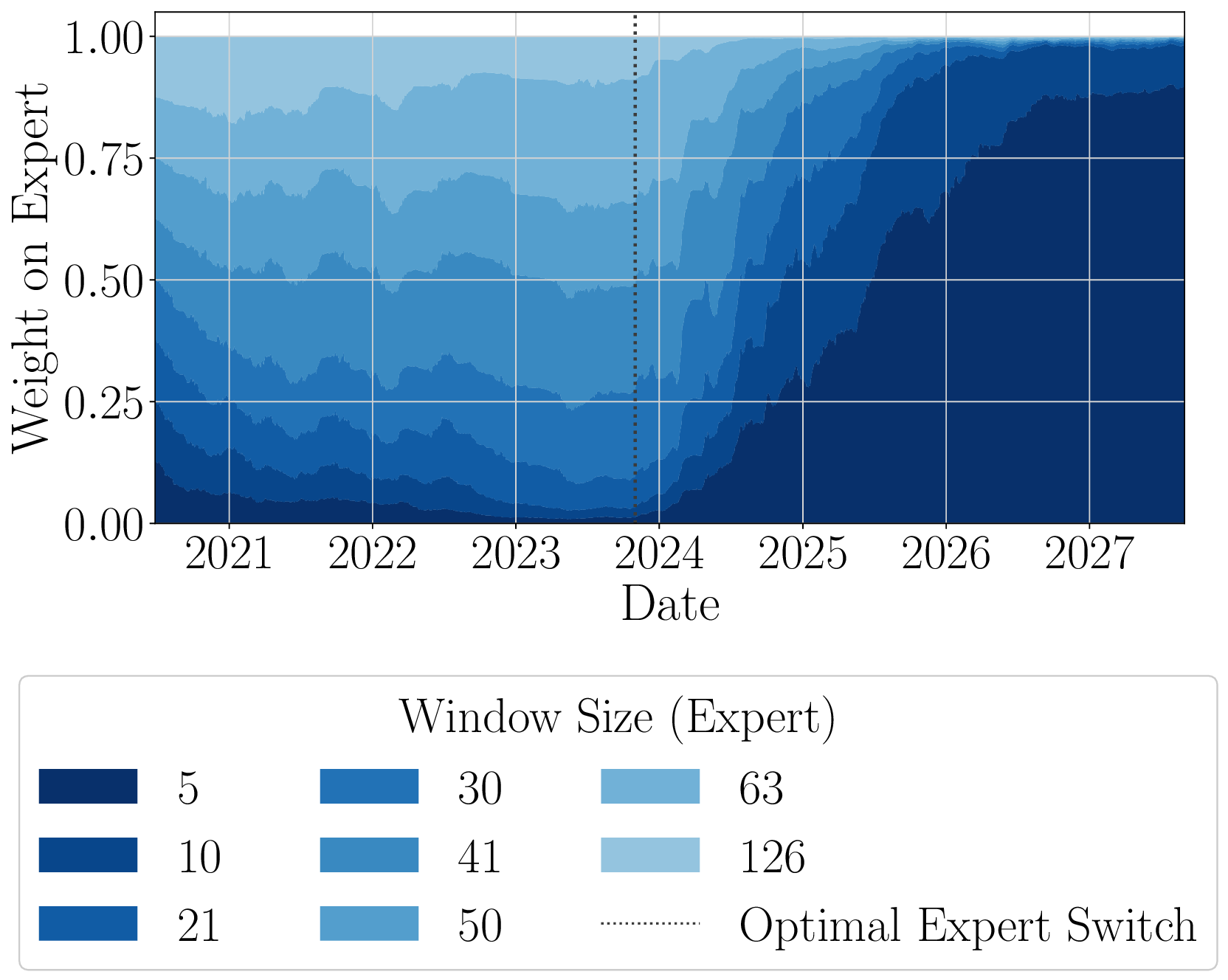}
        \caption{Fixed Share}
        \label{fig: FS q_j syn stacked}
    \end{subfigure}
    \caption{Evolution of the aggregation weights $q_j(t)$ on eight window sizes.}
    \label{fig: q_j syn}
\end{figure}

\clearpage
\subsection{Comparison with Online Portfolio Selection Algorithms}
\label{appendix: portfolio comparisons}
Besides the two-level framework we set, we compare our work with some standard online portfolio selection benchmarks, including UP (\cite{cover1991universal}), EG (\cite{helmbold1998line}), PAMR (\cite{li2012pamr}), CWMR (\cite{li2013confidence}) and OLMAR (\cite{li2014online}. 
    For these comparators, we initialize their parameters to the theoretically optimal values prescribed in their respective foundational papers. Note that these standard configurations were originally derived under frictionless assumptions. We evaluate them under these standard configurations not to claim they are fully optimized for our specific setting, but to illustrate the structural necessity of incorporating cost-sensitive directly into the financial optimization process.
 
Additionally, Figure~\ref{fig: all tracking stock syn} compares our two-level framework against several benchmark algorithms commonly used in online portfolio selection. 
For this comparison, we define $R_{\rm{stock}}(T, K)$ as the tracking regret against the best stock sequence with at most $K$ switches.

Consider the portfolio $\mathcal{M}=\{1, 2, \dots, m\}$ with $m\ge2$. 
We treat each asset $i\in\mathcal{M}$ as an expert. 
Write $\mathbf{i}:=(i_1, \dots, i_T)\in\mathcal{M}^T$ for an arbitrary comparator sequence of experts, where $i_t$ is the expert selected by the comparator at time $t$.
For the switching budget $K$, define 
\begin{align*}
    \mathcal{I}_{T, K}
    :=
    \left\{
    \mathbf{i}\in\mathcal{M}^T:S_T(\mathbf{i})\le K,
    \right\}
\end{align*}
where $S_T(\mathbf i)$ counts the switches in $\mathbf i$.
For a given horizon $T\geq 1$ and a switching budget $K\in\{0,1,\ldots,T-1\}$, the tracking regret $R_{\rm{stock}}$ against stock sequence is defined as
\[
R_{\rm{stock}}(T,K)
:=
\sum_{t=1}^T \ell(\widehat{\mathbf w}(t),\mathbf{y}(t))
-
\min_{\mathbf i\in\mathcal I_{T,K}}
\sum_{t=1}^T \ell(\mathbf w_{i_t}(t),\mathbf{y}(t)).
\]
The minimization over $\mathcal{I}_{T,K}$ selects the best such expert sequence in hindsight.

Consistent with existing literature, the ``experts' advice'' is formulated as a buy-and-hold strategy for individual stocks, and the loss function is the negative logarithmic return. 
Under this configuration, our two-level framework achieves a lower $R_{\rm{stock}}$ than the compared algorithms in this synthetic experiment. 
Negative regret indicates that our framework incurs lower cumulative loss than the best stock sequence with at most $K$ switches.

\begin{figure}[htbp]
    \centering
    \includegraphics[width=.6\linewidth]{ 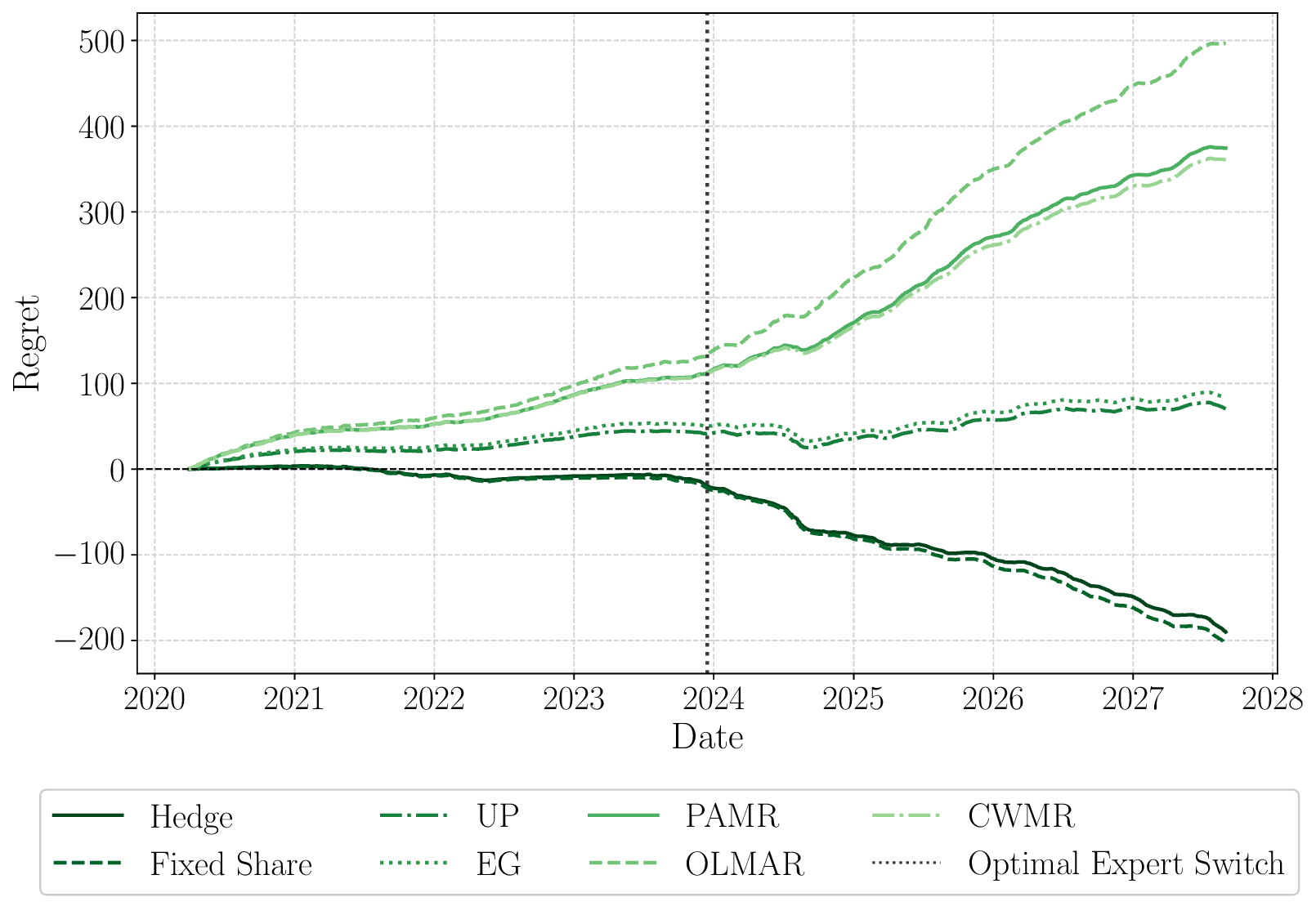}
    \caption{$R_{\rm{stock}}$ for Hedge, Fixed Share, and Comparators with Negative Log Return Loss }
    \label{fig: all tracking stock syn}
\end{figure}

\clearpage
\subsection{Empirical Studies: Dow Jones 30}
\label{subsection: empirical studies: dow 30}
This section presents an empirical study that uses historical data to validate our proposed framework.

\paragraph{The Setup.}
For the portfolio setting, we use the \emph{Dow Jones Industrial Average} and invest directly in its constituent stocks to benchmark performance against the \texttt{DJI} index itself. 
We choose the reference expert set $\mathcal{N}= \{5, 10, 21, 30, 41, 50, 63, 126\}$. 
Because the \texttt{DJI} periodically updates its components, we isolate a period free of any asset additions or removals to simplify our historical data. 
From 2021/03/04 to 2024/02/25, we construct the portfolio using the 30 exact components of the \texttt{DJI} index over that period. 
The portfolio additionally includes \texttt{BIL},
a 1--3 Month Treasury Bill ETF.
The asset tickers are listed in
Appendix~\ref{appendix: Dow Jones 30 tickers}.
We set the transaction cost rate $c=0.001$ (10 bps) for turnover trades.
Since $\alpha^*$ and $\eta^*$ depend on the switching times, we choose a prescribed switching budget $K(T)$ satisfying \eqref{eq: switch time limit}. 
For any constant $\varepsilon>0$, if sufficiently large time horizon $T$ satisfies 
$
\log(T) \big(\log(\log(T))\big)^{1 + \varepsilon}>1,
$
we define the switching budget $K(T)$ as:
$$
K(T) = \left\lfloor\frac{T}{\log(T) \big(\log(\log(T))\big)^{1 + \varepsilon}}\right\rfloor.
$$
Since it follows that:
$$
\lim_{T \to \infty} \frac{\frac{T}{\log(T) (\log(\log(T)))^{1+\varepsilon}}}{\frac{T}{\log(T)\log(\log(T))}} \ = 
\ \lim_{T \to \infty} \frac{1}{\big(\log(\log(T))\big)^\varepsilon}
=0.
$$
In the following studies, we use this $K(T)$ as our switching budget $K$ estimator. 

\paragraph{Performance Evaluation.}

    Figure \ref{fig: H+FS tracking expert 2020-24} presents the performance of Cost-Sensitive Regret evaluated on the Dow Jones Industrial Average (DJIA). 
    As depicted in Figure \ref{fig: q_j 2020-24}, the weight distribution remains largely uniform, with the proportion assigned to $j=126$ being marginally higher than that of the others. 
    Figure \ref{fig: all tracking stock 2020-24} confirms that the $R_{\rm{stock}}$ of our proposed method remains consistently lower than that of the comparative algorithms. 
    Furthermore, Table~\ref{table: performance-1 2020-24} and Figure~\ref{fig: H+FS vt 2020-24} illustrate the trading performance of our method relative to the baselines.  
    Figure~\ref{fig: H+FS vt 2020-24} demonstrates that our two-level framework progressively accumulates a higher account value, even when subjected to these transaction cost penalties. 

\begin{figure}[htbp]
    \centering
    \includegraphics[width=.6\linewidth]{ 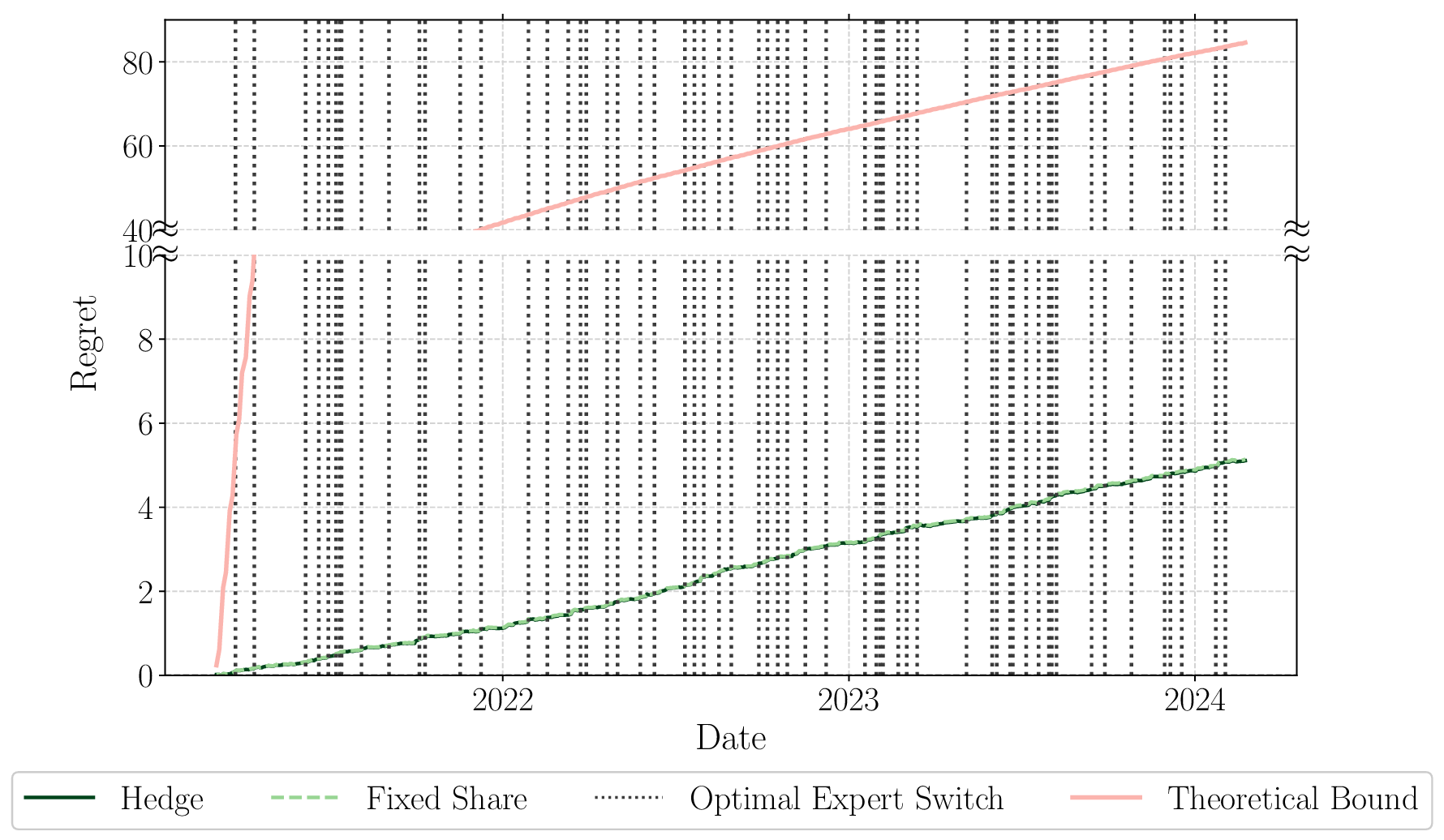}
    \caption{Cost-Sensitive Tracking Regret for Hedge and Fixed Share with Negative PnL Loss (2021/03/04-2024/02/25).}
    \label{fig: H+FS tracking expert 2020-24}
\end{figure}

\begin{figure}[htbp]
    \centering
    \begin{subfigure}[b]{0.45\textwidth}
        \centering
        \includegraphics[width=\textwidth]{ 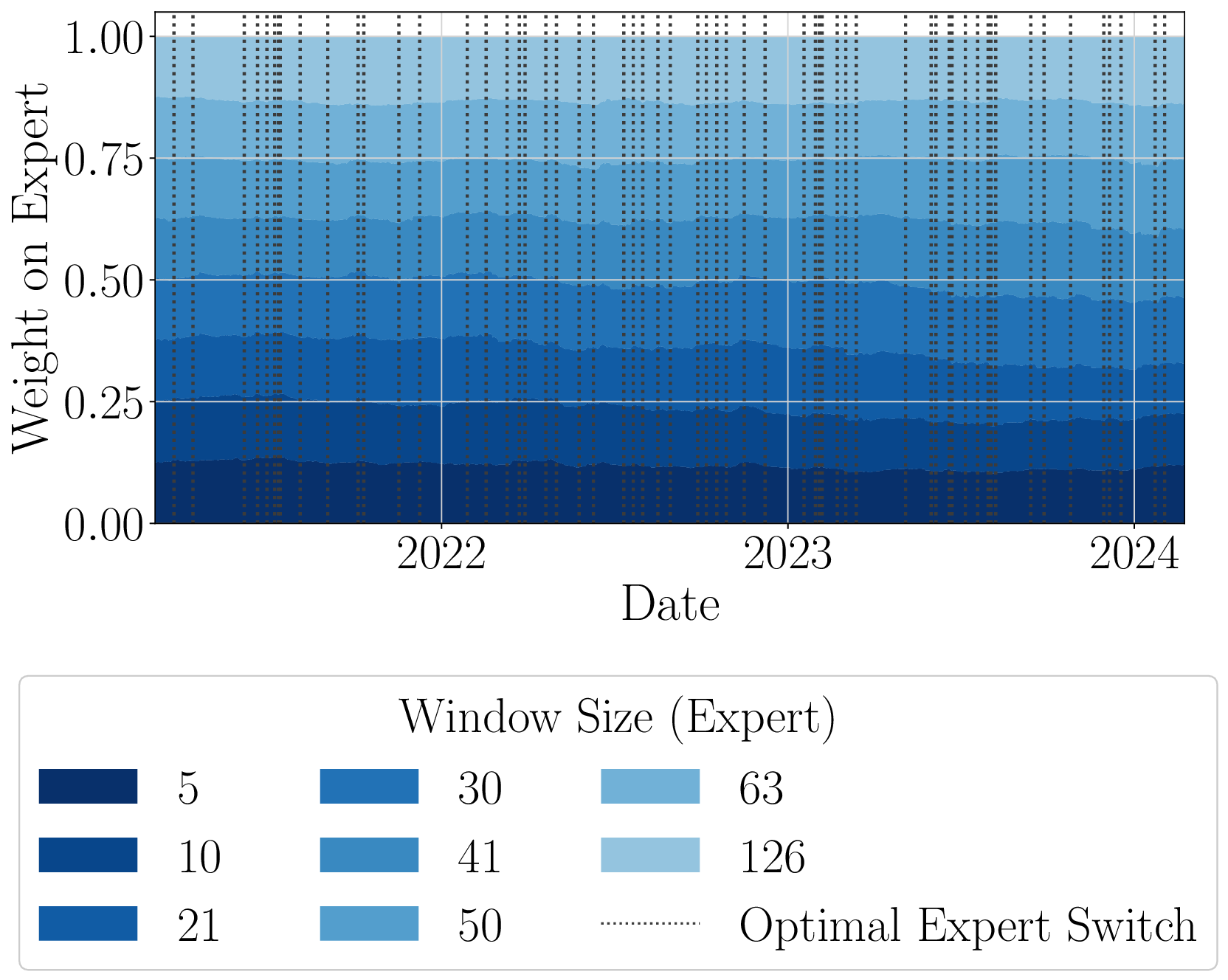}
        \caption{Hedge}
        \label{fig: H q_j 2020-24}
    \end{subfigure}
    \hfill
    \begin{subfigure}[b]{0.45\textwidth}
        \centering
        \includegraphics[width=\textwidth]{ 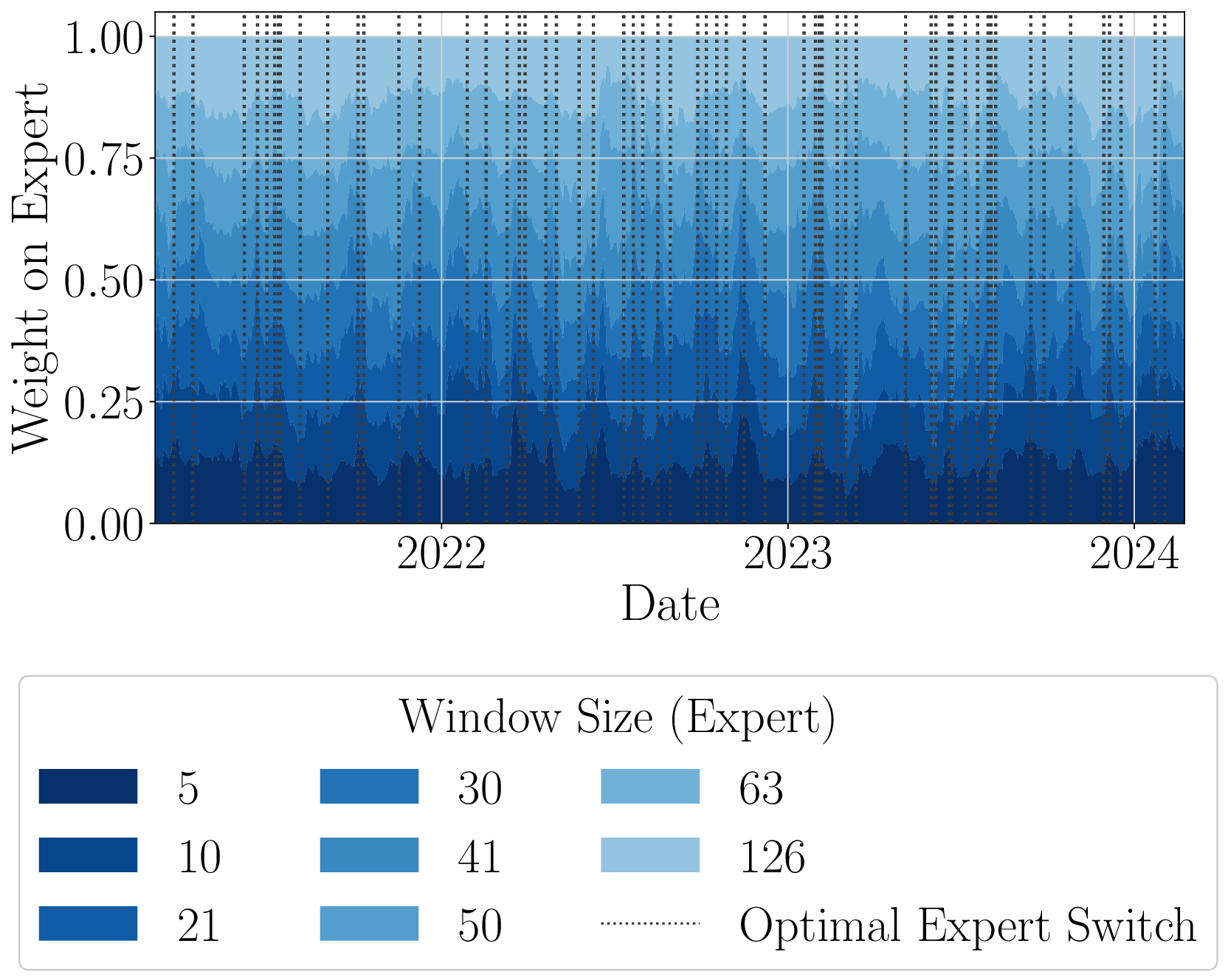}
        \caption{Fixed Share}
        \label{fig: FS q_j 2020-24}
    \end{subfigure}
    \caption{Evolution of the aggregation weights $q_j(t)$ assigned to window-size experts (2021/03/04-2024/02/25).}
    \label{fig: q_j 2020-24}
\end{figure}

\begin{figure}[htbp]
    \centering
    \includegraphics[width=.6\linewidth]{ 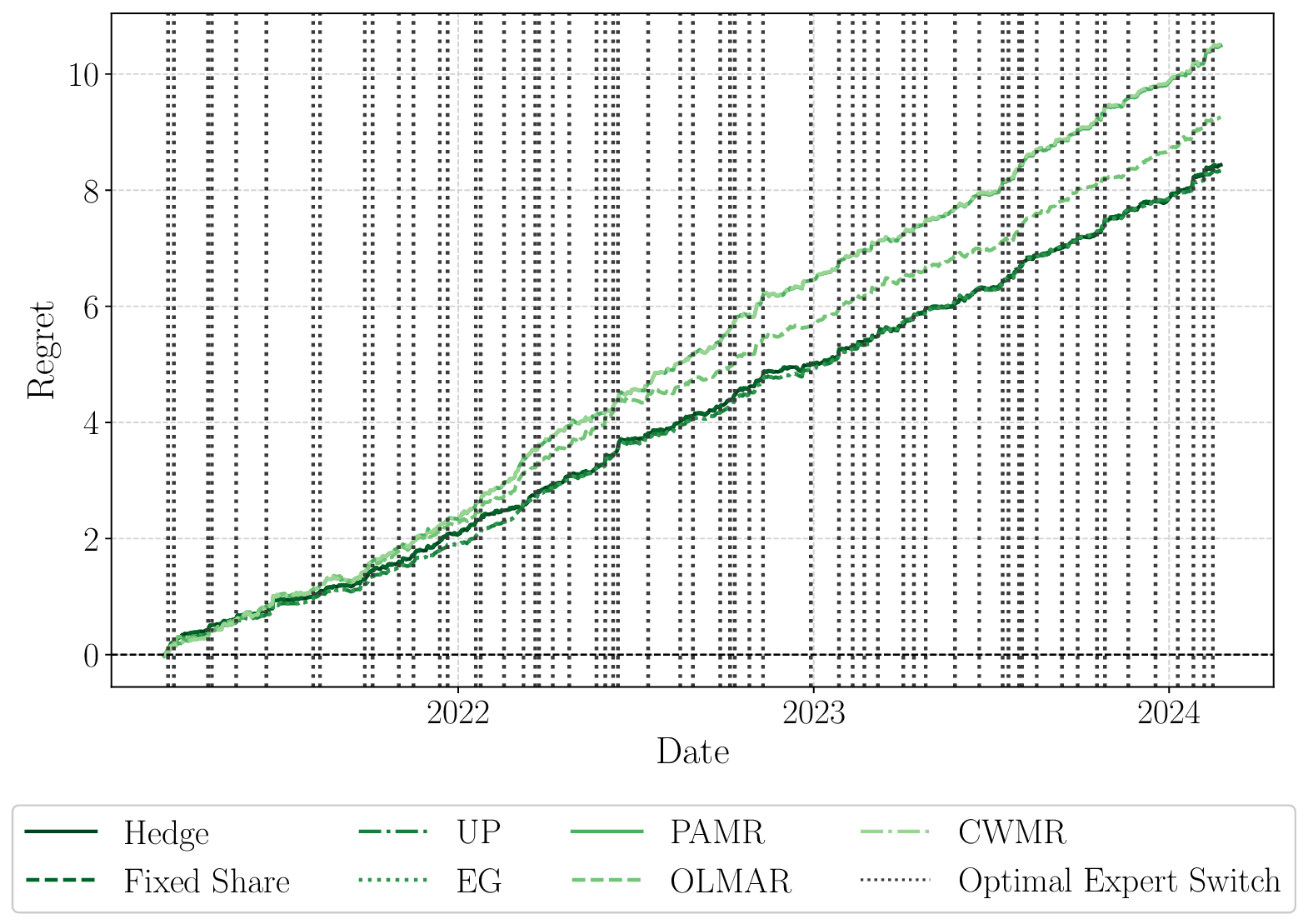}
    \caption{$R_{\rm{stock}}$ for Hedge, Fixed Share, and Comparators with Negative Log Return Loss (2021/03/04-2024/02/25).}
    \label{fig: all tracking stock 2020-24}
\end{figure}

\begin{figure}[htbp]
    \centering
    \includegraphics[width=.6\linewidth]{ 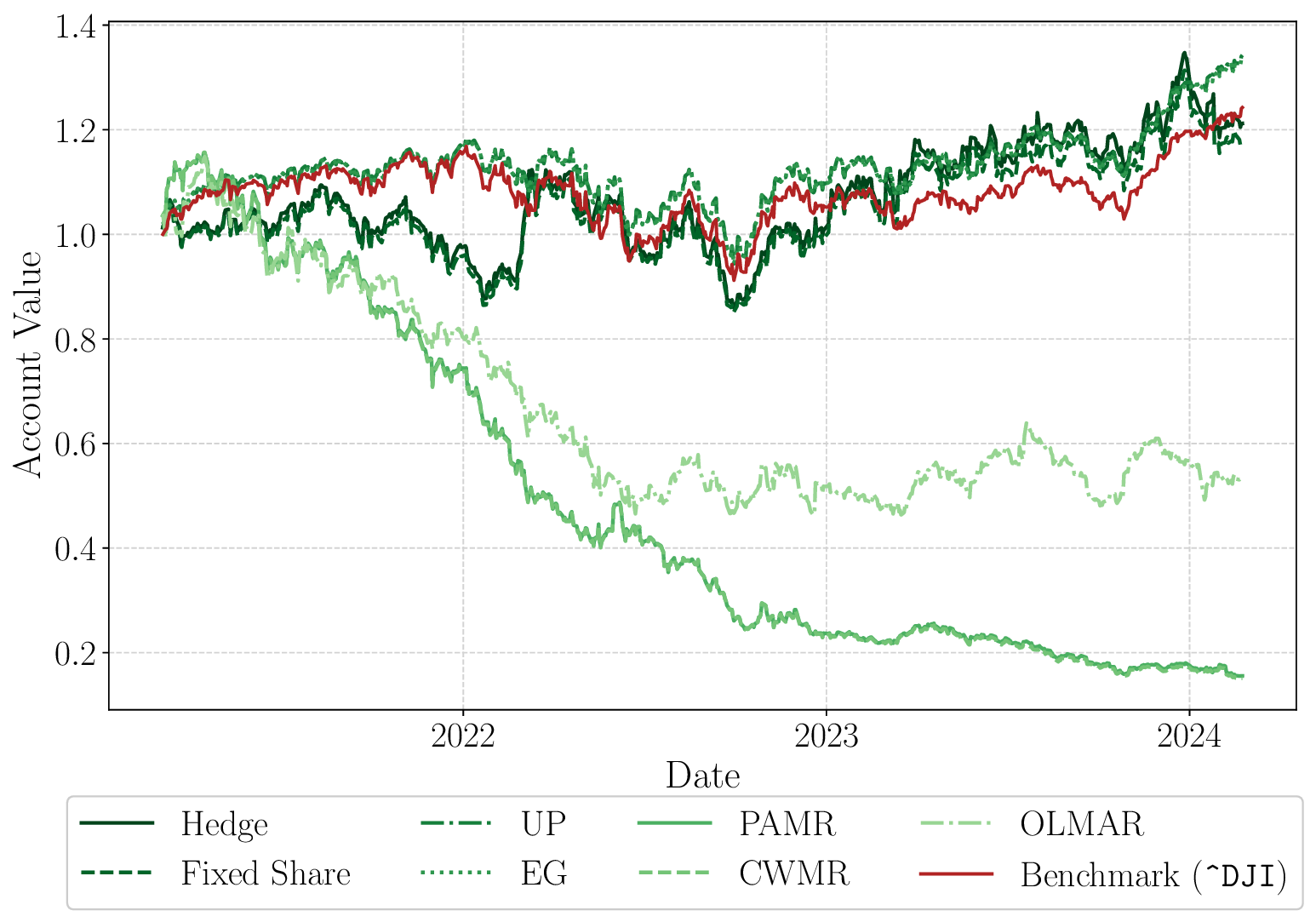}
    \caption{Account Value for Hedge, Fixed Share and benchmark \texttt{DJI}(2021/03/04-2024/02/25).}
    \label{fig: H+FS vt 2020-24}
\end{figure}

\begin{table}[htbp]
    \caption{Trading Performance Metric for different Window Sizes (2021/03/04-2024/02/25).}
    \centering
    \scriptsize
    \begin{tabular}{l r r r r r}
        \toprule
        & 5-day & 10-day & 21-day & 30-day & 41-day\\
        \midrule
        Cumulative Return(\%)  & 98.3533 & 23.8690 & 1.2471 & 45.2568 & 56.7585 \\
        Sample Mean of Return(\%)  & 0.1068 & 0.0426 & 0.0141 & 0.0646 & 0.0734 \\
        Sample Standard Deviation & 0.0175 & 0.0167 & 0.0158 & 0.0172 & 0.0164 \\
        Annualized Sharpe Ratio  & 0.8862 & 0.3170 & 0.0487 & 0.5119 & 0.6226 \\
        Maximum Drawdown(\%)  & -28.9273 & -29.1735 & -33.2589 & -29.5701 & -23.8111 \\
        \bottomrule
    \end{tabular}
    \label{table: performance-1 2020-24}
\end{table}
\begin{table}[htbp]
    \ContinuedFloat
    \caption{Trading performance metrics (continued): longer-window experts, Hedge, and Fixed Share.}
    \centering
    \scriptsize
    \begin{tabular}{l r r r r r}
        \toprule
         & 50-day & 63-day & 126-day & Hedge & Fixed Share\\
        \midrule
        Cumulative Return(\%) & 52.7662 & 9.4500 & 38.8086 & 21.1950 & 17.7014 \\
        Sample Mean of Return(\%) & 0.0719 & 0.0274 & 0.0589 & 0.0338 & 0.0301 \\
        Sample Standard Deviation & 0.0174 & 0.0174 & 0.0173 & 0.0127 & 0.0129 \\
        Annualized Sharpe Ratio & 0.5705 & 0.1651 & 0.4551 & 0.3060 & 0.2568 \\
        Maximum Drawdown(\%) & -29.8611 & -37.9590 & -28.2687 & -23.5123 & -23.0759 \\
        \bottomrule
    \end{tabular}
\end{table}
\begin{table}[htbp]
    \ContinuedFloat
    \caption{Trading performance metrics (continued): online portfolio selection baselines.}
    \centering
    \scriptsize
    \begin{tabular}{l c c c c c}
        \toprule
        & UP & EG & PAMR & CWMR & OLMAR\\
        \midrule
        Cumulative Return(\%) & 48.7250 & 48.2584 & -88.2441 & -88.5860 & -38.9035 \\
        Sample Mean of Return(\%) & 0.0497 & 0.0493 & -0.2262 & -0.2296 & -0.0365 \\
        Sample Standard Deviation & 0.0093 & 0.0092 & 0.0190 & 0.0190 & 0.0200 \\
        Annualized Sharpe Ratio & 0.7165 & 0.7109 & -1.9556 & -1.9840 & -0.3535 \\
        Maximum Drawdown(\%) & -20.2748 & -20.2783 & -88.9822 & -89.3036 & -61.0332 \\
        \bottomrule
    \end{tabular}
\end{table}

\clearpage
\subsection{Empirical Studies: S\&P 500}
\label{subsection: emprical_studies_SP500}
\paragraph{The Setup.}
    Keep the expert setting $\mathcal{N}$ and the transaction cost rate as above. For the portfolio setting, we use the \emph{S\&P 500} (Ticker: \texttt{GSPC}) and invest directly in its constituent stocks to benchmark performance against the \texttt{GSPC} index itself. 
    Because the \texttt{GSPC} periodically updates its components, from 2020/01/02 to 2026/08/01, we remove all the changed stocks and construct our portfolio using the 481 components of the \texttt{GSPC} during that window (Tickers are listed in Appendix~\ref{appendix: SP500 tickers}.) and the risk-free asset \texttt{BIL}, which represents the 1-3 Month T-Bill ETF.

\paragraph{Performance Evaluation.}
    
    Figure \ref{fig: H+FS tracking expert 2020-26} displays the performance of the Cost-Sensitive Regret evaluated on the S\&P 500 index. Consistent with previous observations, Figure \ref{fig: q_j 2020-26} shows a generally uniform weight allocation, with the proportion for $j=126$ remaining slightly elevated. 
    Figure \ref{fig: all tracking stock 2020-26} indicates that our method continues to yield a marginally lower $R_{\rm{stock}}$ compared to the baseline algorithms, while Figure \ref{fig: H+FS vt 2020-26} demonstrates that our method achieves a higher overall account value, and Table \ref{table: performance-1 2020-26} concludes the financial performance.
    
    It is important to note that these empirical studies are intended solely to demonstrate the practical application of our framework over a restricted timeframe, utilizing Hedge and Fixed Share as illustrative examples to account for window sizes and transaction costs. The primary objective is to underscore the necessity of incorporating such considerations, rather than to assert any predictive capability regarding the future performance of specific financial instruments.

\begin{figure}[htbp]
    \centering
    \includegraphics[width=.8\linewidth]{ 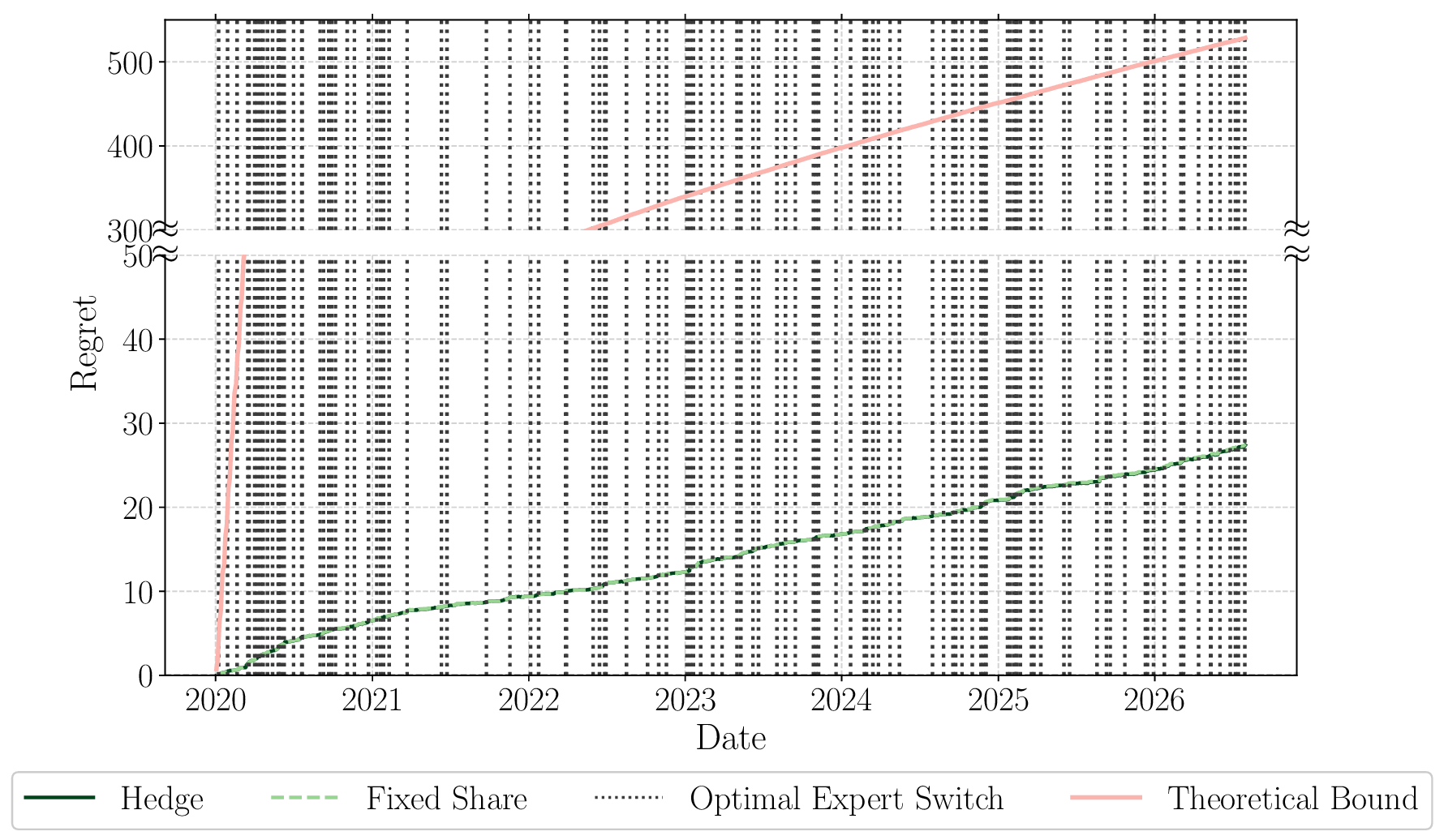}
    \caption{Cost-Sensitive Tracking Regret of Expert for Hedge and Fixed Share with Negative PnL Loss (2020/01/02-2026/08/01).}
    \label{fig: H+FS tracking expert 2020-26}
\end{figure}

\begin{figure}[htbp]
    \centering
    \begin{subfigure}[b]{0.48\textwidth}
        \centering
        \includegraphics[width=\textwidth]{ 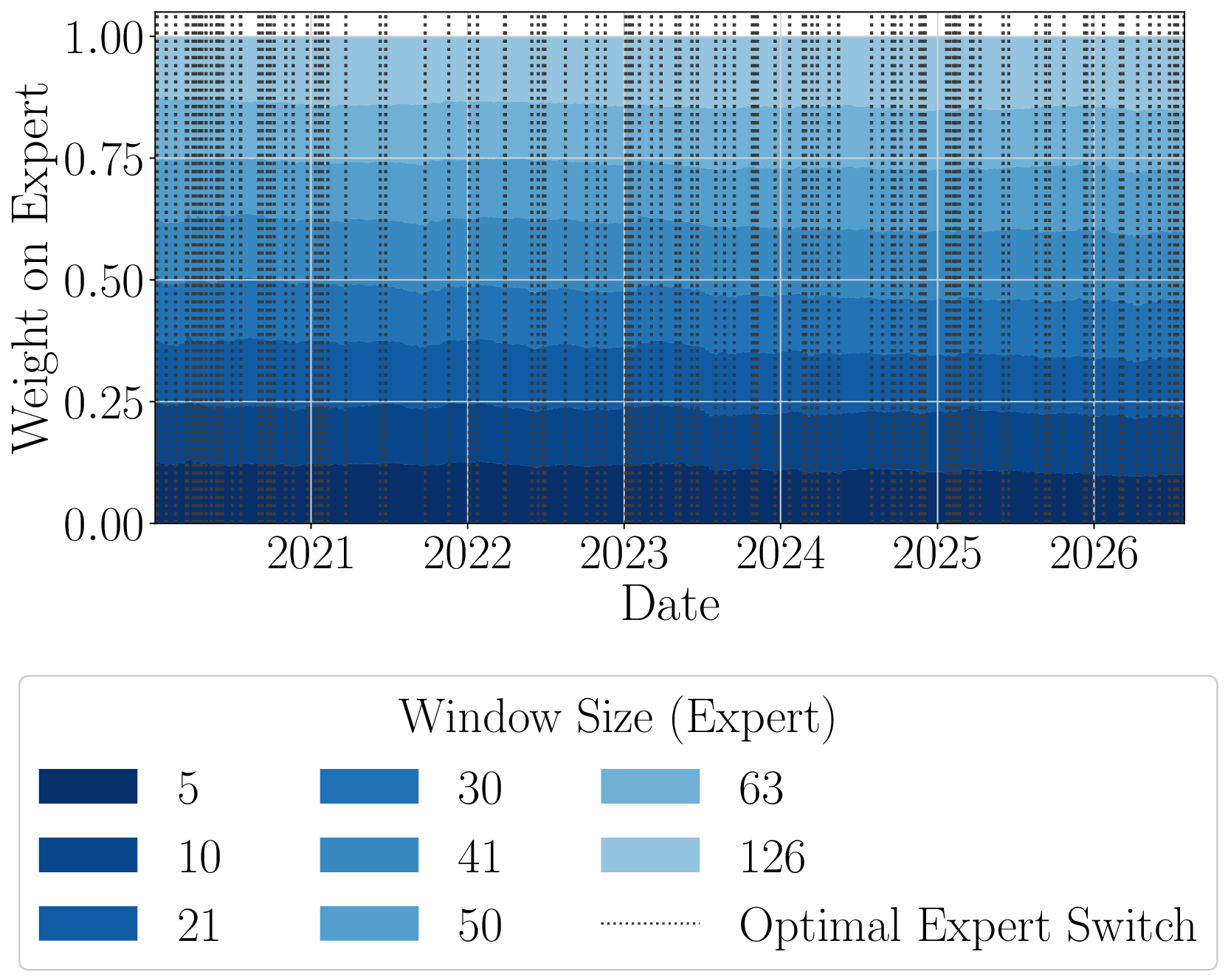}
        \caption{Hedge}
        \label{fig: H q_j 2020-26}
    \end{subfigure}
    \hfill
    \begin{subfigure}[b]{0.48\textwidth}
        \centering
        \includegraphics[width=\textwidth]{ 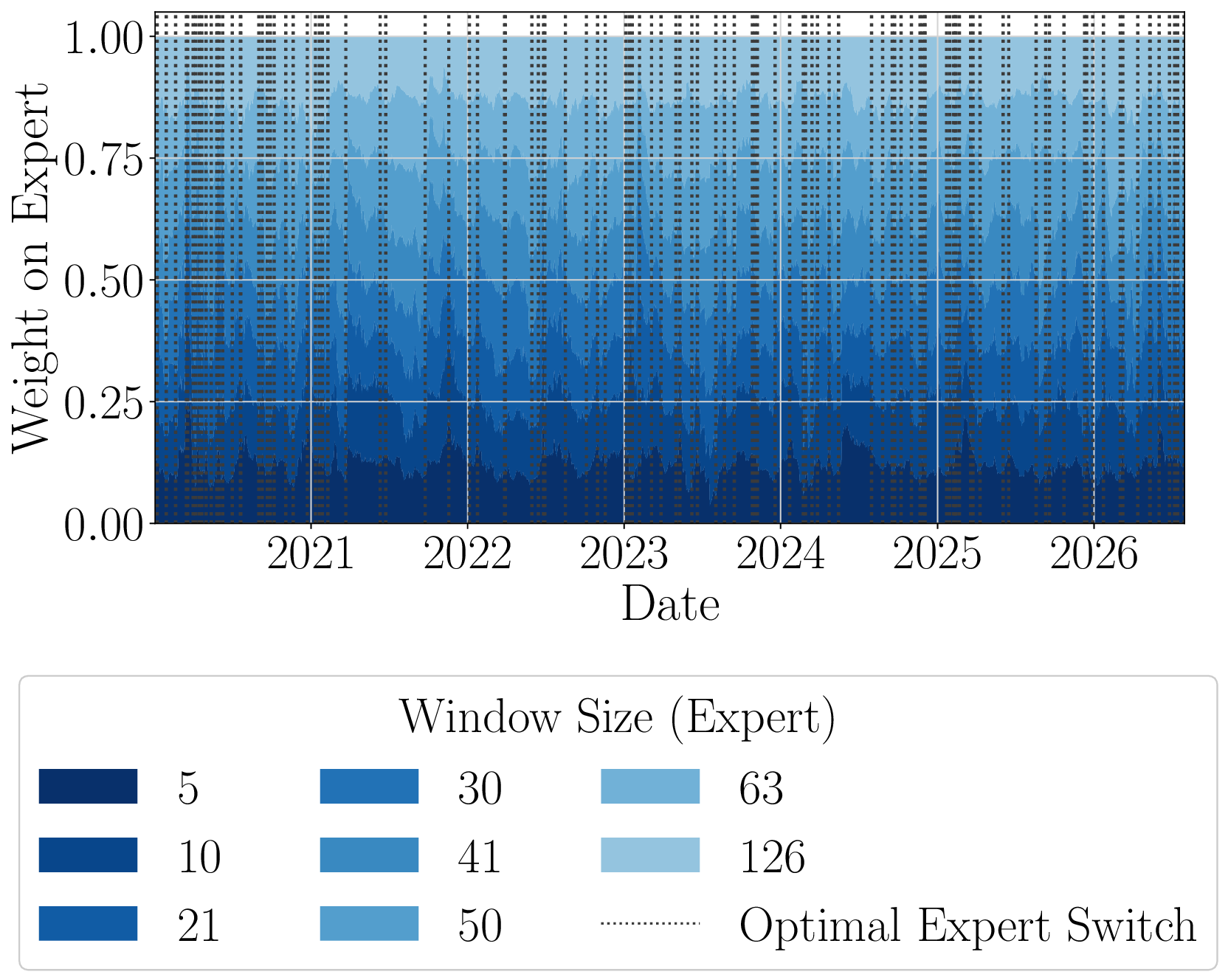}
        \caption{Fixed Share}
        \label{fig: FS q_j 2020-26}
    \end{subfigure}
    \caption{Evolution of the aggregation weights $q_j(t)$ assigned to window-size experts (2020/01/02-2026/08/01).}
    \label{fig: q_j 2020-26}
\end{figure}

\begin{figure}[htbp]
    \centering
    \includegraphics[width=.8\linewidth]{ 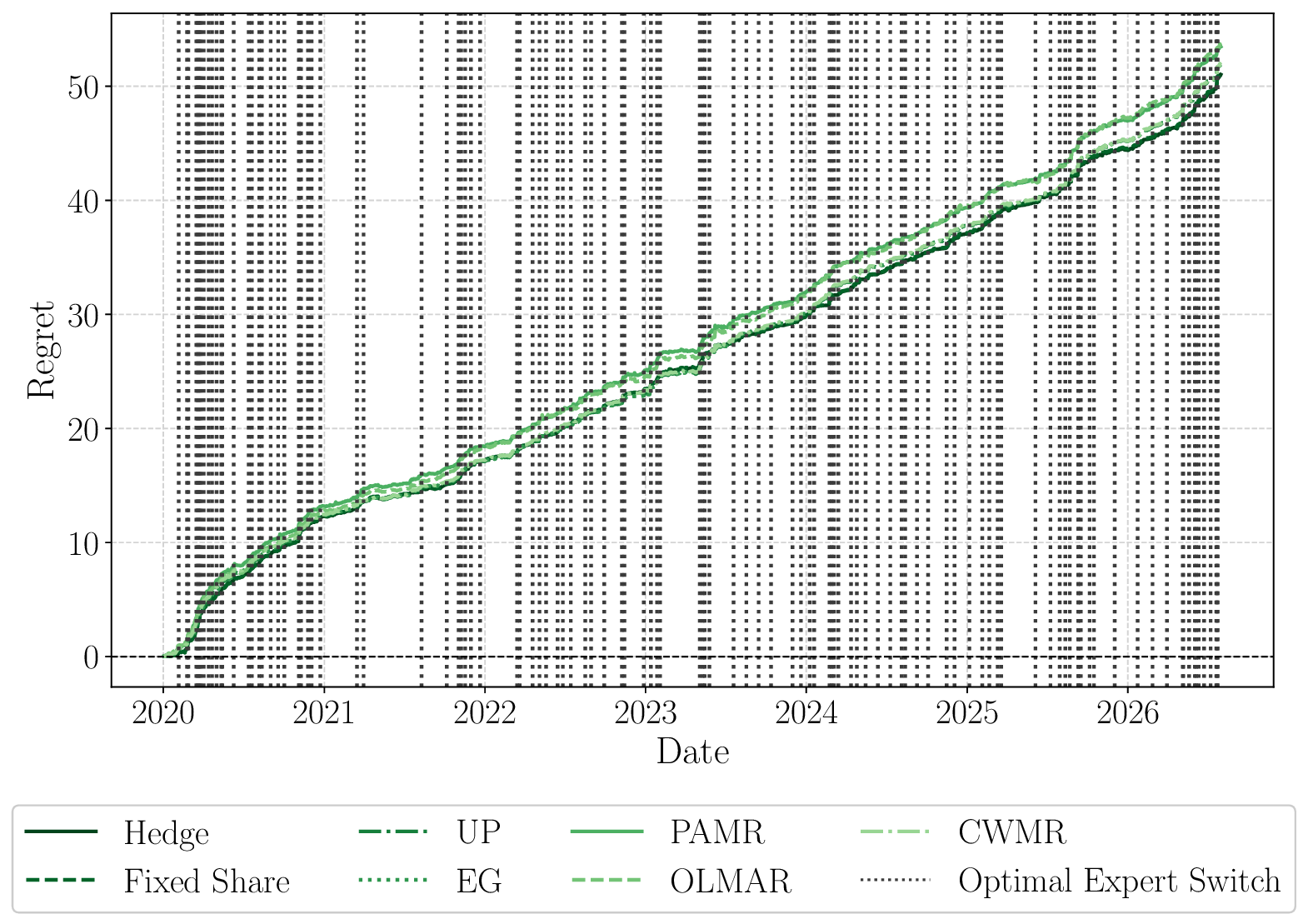}
    \caption{$R_{\rm{stock}}$ for Hedge, Fixed Share, and Comparators with Negative Log Return Loss (2020/01/02-2026/08/01). }
    \label{fig: all tracking stock 2020-26}
\end{figure}

\begin{figure}[htbp]
    \centering
    \includegraphics[width=.8\linewidth]{ 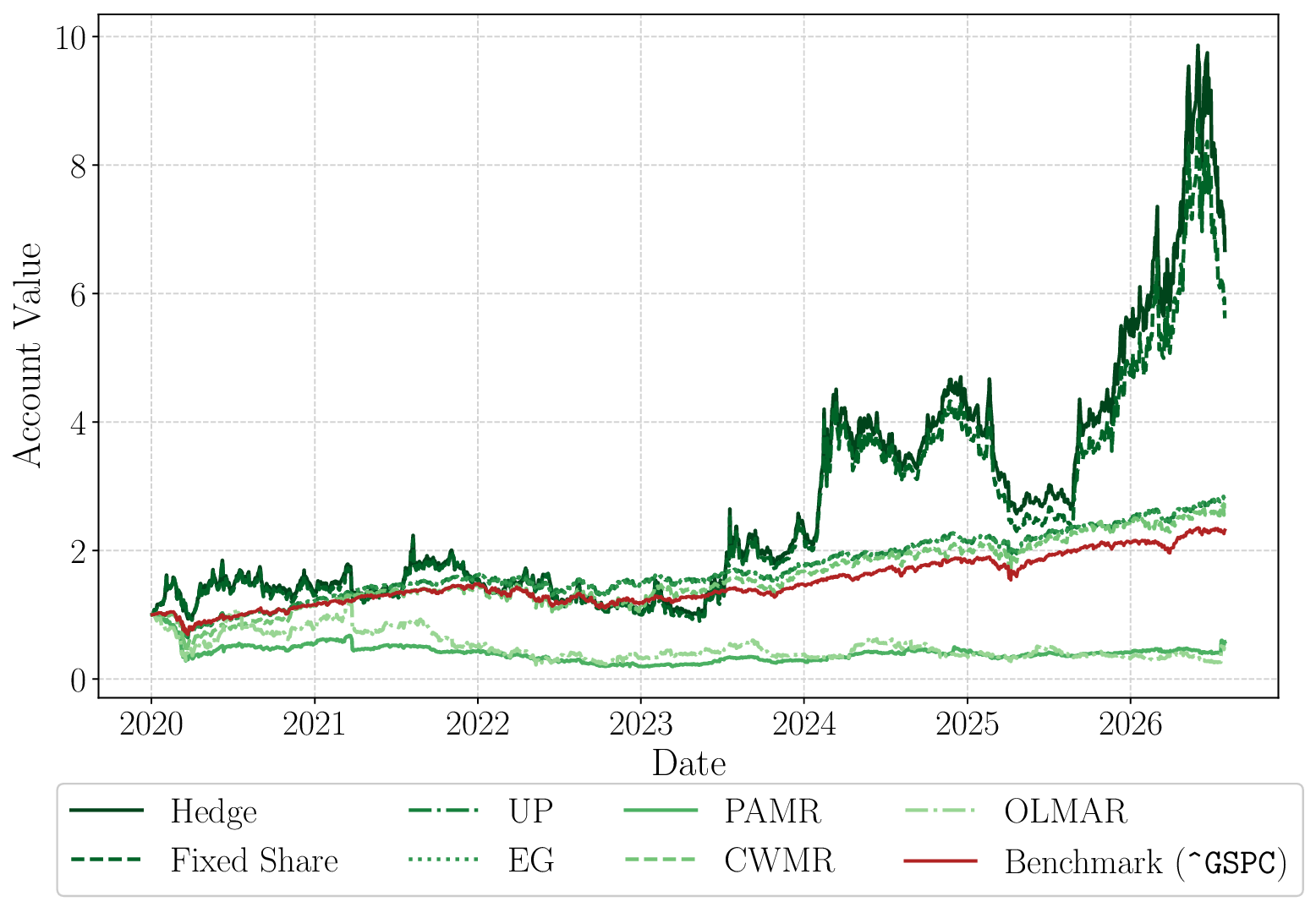}
    \caption{Account Value for Hedge, Fixed Share and benchmark \texttt{GSPC}(2020/01/02-2026/08/01).}
    \label{fig: H+FS vt 2020-26}
\end{figure}

\begin{table}[htbp]
    \caption{Trading Performance Metric for different Window Sizes (2020/01/02-2026/08/01).}
    \centering
    \scriptsize
    \begin{tabular}{l r r r r r}
        \toprule
        & 5-day & 10-day & 21-day & 30-day & 41-day\\
        \midrule
        Cumulative Return(\%)  & 256.0875 & 638.5071 & 361.6039 & 261.9927 & 1046.4484 \\
        Sample Mean of Return(\%)  & 0.1840 & 0.2311 & 0.2013 & 0.1857 & 0.2545 \\
        Sample Standard Deviation & 0.0468 & 0.0476 & 0.0473 & 0.0469 & 0.0464 \\
        Annualized Sharpe Ratio  & 0.5875 & 0.7352 & 0.6400 & 0.5924 & 0.8333 \\
        Maximum Drawdown(\%)  & -62.5065 & -63.2920 & -74.2220 & -76.1463 & -67.7831 \\
        \bottomrule
    \end{tabular}
    \label{table: performance-1 2020-26}
\end{table}
\begin{table}[htbp]
    \ContinuedFloat
    \caption{Trading performance metrics (continued): longer-window experts, Hedge, and Fixed Share.}
    \centering
    \scriptsize
    \begin{tabular}{l r r r r r}
        \toprule
         & 50-day & 63-day & 126-day & Hedge & Fixed Share\\
        \midrule
        Cumulative Return(\%) & 385.4326 & 297.5403 & 1912.9062 & 567.2010 & 461.4739 \\
        Sample Mean of Return(\%) & 0.1995 & 0.1823 & 0.2784 & 0.1806 & 0.1728 \\
        Sample Standard Deviation & 0.0460 & 0.0445 & 0.0441 & 0.0364 & 0.0372 \\
        Annualized Sharpe Ratio & 0.6512 & 0.6117 & 0.9642 & 0.7400 & 0.6918 \\
        Maximum Drawdown(\%) & -78.5995 & -67.4449 & -60.4583 & -56.9601 & -58.5937 \\
        \bottomrule
    \end{tabular}
\end{table}
\begin{table}[htbp]
    \ContinuedFloat
    \caption{Trading performance metrics (continued): online portfolio selection baselines.}
    \centering
    \scriptsize
    \begin{tabular}{l c c c c c}
        \toprule
        & UP & EG & PAMR & CWMR & OLMAR\\
        \midrule
        Cumulative Return(\%) & 205.3887 & 205.7900 & -52.0775 & 205.5993 & -44.0283 \\
        Sample Mean of Return(\%) & 0.0712 & 0.0712 & 0.0094 & 0.0827 & 0.0826 \\
        Sample Standard Deviation & 0.0129 & 0.0129 & 0.0321 & 0.0199 & 0.0501 \\
        Annualized Sharpe Ratio & 0.7456 & 0.7468 & -0.0051 & 0.5775 & 0.2290 \\
        Maximum Drawdown(\%) & -38.3346 & -38.2862 & -84.5912 & -59.5749 & -81.6301 \\
        \bottomrule
    \end{tabular}
\end{table}

\clearpage
\section{Assets Used in the Empirical Studies}

\paragraph{Dow Jones 30.}
\label{appendix: Dow Jones 30 tickers}
The tickers of the Dow Jones 30 stocks considered in Section~\ref{subsection: empirical studies: dow 30} are listed below:
\texttt{WBA, CRM, HON, AMGN, DOW, AAPL, GS, V, NKE, UNH, TRV, CSCO, CVX, VZ, MSFT, HD, INTC, JNJ, WMT, CAT, JPM, DIS, BA, KO, MCD, AXP, IBM, MRK, MMM, PG} and the risk-free asset \texttt{BIL}, which represents the 1-3 Month T-Bill ETF.

\paragraph{S\&P 500.}
\label{appendix: SP500 tickers}
The tickers of the S\&P 500 stocks considered in Section~\ref{subsection: emprical_studies_SP500} are listed below:

\texttt{A, AAPL, ABBV, ABT, ACGL, ACN, ADBE, ADI, ADM, ADP, ADSK, AEE, AEP, AES, AFL, AIG, AIZ, AJG, AKAM, ALB, ALGN, ALL, ALLE, AMAT, AMCR, AMD, AME, AMGN, AMP, AMT, AMZN, ANET, AON, AOS, APA, APD, APH, APO, APTV, ARE, ARES, ATO, AVGO, AVY, AWK, AXON, AXP, AZO, BA, BAC, BALL, BAX, BBY, BDX, BEN, BF-B, BG, BIIB, BKNG, BKR, BLDR, BLK, BMY, BNY, BR, BRK-B, BRO, BSX, BWA, BX, BXP, C, CAG, CAH, CARR, CAT, CB, CBOE, CBRE, CCI, CCL, CDNS, CDW, CE, CEG, CF, CFG, CHD, CHRW, CHTR, CI, CINF, CL, CLX, CMA, CMCSA, CME, CMG, CMI, CMS, CNC, CNP, COF, COO, COP, COR, COST, CPB, CPRT, CPT, CRWD, CSCO, CSGP, CSX, CTAS, CTC, CTRA, CTSH, CTVA, CVS, CVX, CZR, D, DAL, DD, DE, DFS, DG, DGX, DHI, DHR, DIS, DOC, DOV, DOW, DPZ, DRI, DTE, DUK, DVA, DVN, DXCM, EA, EBAY, ECL, ED, EFX, EG, EIX, EL, ELV, EMN, EMR, ENPH, EOG, EPAM, EQT, ERIE, ES, ESS, ETN, ETR, EVRG, EW, EXC, EXPD, EXPE, EXR, F, FANG, FAST, FCX, FDS, FDX, FE, FFIV, FI, FICO, FIS, FITB, FLT, FMC, FOX, FOXA, FRT, FSLR, FTNT, GE, GEHC, GEV, GILD, GIS, GL, GLW, GM, GNRC, GOOG, GOOGL, GPC, GPN, GRMN, GS, GWW, HAL, HAS, HBAN, HCA, HD, HES, HIG, HII, HLT, HOLX, HON, HPE, HPQ, HRL, HSIC, HST, HSY, HUBB, HUM, HWM, IBM, ICE, IDXX, IEX, IFF, ILMN, INCY, INTC, INTU, INVU, IP, IPG, IQV, IR, IRM, ISRG, IT, ITW, IVZ, J, JBHT, JCI, JKHY, JNJ, JNPR, JPM, K, KDP, KEY, KEYS, KHC, KIM, KLAC, KMB, KMI, KMX, KO, KR, KKV, L, LDOS, LEN, LH, LHX, LIN, LKQ, LLY, LMT, LNT, LOW, LRCX, LUV, LVS, LYB, LYV, MA, MAA, MAR, MAS, MCD, MCHP, MCK, MCO, MDLZ, MDT, MET, META, MGM, MHK, MKC, MKTX, MLM, MMC, MMM, MNST, MO, MOH, MOS, MPC, MPWR, MRK, MRO, MS, MSCI, MSFT, MSI, MTB, MTCH, MTD, MU, NCLH, NDAQ, NDSN, NEE, NEM, NFLX, NI, NKE, NOC, NOW, NRG, NSC, NTAP, NTRS, NUE, NVDA, NVR, NWS, NWSA, NXPI, O, ODFL, OKE, OMC, ON, OKE, ORCL, ORLY, OTIS, OXY, PANW, PARA, PAYC, PAYX, PCAR, PCG, PEAK, PEG, PEP, PFE, PFG, PG, PGR, PH, PHM, PKG, PLD, PM, PNC, PNR, PNW, PODD, POOL, PPG, PPL, PRU, PSA, PTC, PTEN, PWR, PXD, PYPL, QCOM, QRVO, RCL, REG, REGN, RF, RHI, RJF, RL, RMD, ROK, ROL, ROP, ROST, RSG, RTX, RVTY, SBUX, SCHW, SHW, SJM, SLB, SMCI, SNA, SNPS, SO, SPG, SPGI, SRE, STE, STLD, STT, STX, STZ, SWK, SWKS, SYF, SYK, SYY, T, TAP, TARE, TDG, TDY, TECH, TEL, TER, TFC, TFX, TGT, TJX, TMO, TMUS, TPR, TRMB, TROW, TRV, TSCO, TSLA, TSN, TT, TTWO, TXN, TXT, TYL, UAL, UBER, UDR, UHS, ULTA, UNH, UNP, UPS, URI, USB, V, VLO, VLTO, VMC, VRSK, VRSN, VRTX, VTR, VZ, WAB, WAT, WBA, WBD, WDC, WEC, WELL, WFC, WHR, WM, WMB, WMT, WRB, WST, WTW, WY, WYNN, XEL, XOM, XYL, XYZ, YUM, ZBH, ZBRA, ZTS})
and the risk-free asset \texttt{BIL}, which represents the 1-3 Month T-Bill ETF.

\end{document}